\documentclass[a4paper,12pt,reqno]{amsart}
\usepackage{amsmath,amsfonts,amsthm,amssymb,color}
\usepackage[T1]{fontenc}
\usepackage{enumerate}

\usepackage{pdfsync}

\newcommand{\bb}{\mathbf{B}}
\newcommand{\be}{\mathbf{E}}

\newcommand{\bp}{\mathbf{P}}

\newcommand{\bx}{{\bf x}}

\newcommand{\der}{\delta}
\newcommand{\di}{\diamond}
\newcommand{\dom}{\mbox{Dom}}

\newcommand{\ka}{\kappa}
\newcommand{\id}{\mbox{Id}}

\newcommand{\ist}{\int_{s}^{t}}
\newcommand{\norm}[1]{\lVert #1\rVert}

\newcommand{\ott}{[0,T]}

\newcommand{\xd}{\mathbf{x}^{\mathbf{2}}}

\newcommand{\1}{{\bf 1}}
\newcommand{\2}{{\bf 2}}

\def\hatotimes{{\hat {\otimes}}}
\def\sumq{\sum_{q=0}^{n-1}}
\newcommand{\bk}{\mathbf{k}}
\newcommand{\bl}{\mathbf{l}}
\newcommand{\bm}{\mathbf{m}}
\newcommand{\bn}{\mathbf{n}}
\newcommand{\bq}{\mathbf{q}}
\newcommand{\bnn}{\mathbf{N}}

\def\EE{\mathbf{E}}

\newcommand{\R}{\mathbb R}
\newcommand{\N}{\mathbb N}

\newcommand{\cb}{\mathcal B}
\newcommand{\cac}{\mathcal C}

\newcommand{\ce}{\mathcal E}
\newcommand{\cf}{\mathcal F}
\newcommand{\ch}{\mathcal H}

\newcommand{\cj}{\mathcal J}
\newcommand{\cl}{\mathcal L}
\newcommand{\cn}{\mathcal N}
\newcommand{\cq}{\mathcal Q}
\newcommand{\cs}{\mathcal S}
\newcommand{\ct}{\mathcal T}

\newcommand{\cz}{\mathcal Z}

\newcommand{\al}{\alpha}
\newcommand{\ep}{\varepsilon}

\newcommand{\ga}{\gamma}

\newcommand{\la}{\lambda}
\newcommand{\laa}{\Lambda}

\newcommand{\oom}{\Omega}

\newcommand{\si}{\sigma}

\newcommand{\vp}{\varphi}
\newcommand{\ze}{\zeta}

\newcommand{\lp}{\left(}
\newcommand{\rp}{\right)}
\newcommand{\lc}{\left[}
\newcommand{\rc}{\right]}
\newcommand{\lcl}{\left\{}
\newcommand{\rcl}{\right\}}
\newcommand{\lln}{\left|}
\newcommand{\rrn}{\right|}
\newcommand{\lla}{\left\langle}
\newcommand{\rra}{\right\rangle}

\newcommand{\fin}
{ \vspace{-0.6cm}
\begin{flushright}
\mbox{$\Box$}
\end{flushright}
\noindent }

\newtheorem{theorem}{Theorem}[section]

\newtheorem{corollary}[theorem]{Corollary}

\newtheorem{definition}[theorem]{Definition}

\newtheorem{hypothesis}[theorem]{Hypothesis}
\newtheorem{lemma}[theorem]{Lemma}

\newtheorem{proposition}[theorem]{Proposition}

\theoremstyle{remark}
\newtheorem{remark}[theorem]{Remark}
\theoremstyle{remark}
\newtheorem{example}[theorem]{Example}

\newcommand{\bean}{\begin{eqnarray*}}
\newcommand{\eean}{\end{eqnarray*}}
\newcommand{\ben}{\begin{enumerate}}
\newcommand{\een}{\end{enumerate}}
\newcommand{\beq}{\begin{equation}}
\newcommand{\eeq}{\end{equation}}

\newcommand{\iott}{\int_0^T}
\newcommand{\hnj}{H_n\big(x(\varphi)\big)}
\newcommand{\hnju}{H_{n-1}\big(x(\varphi)\big)}

\newcommand{\crj}{{}^{[j]}}
\newcommand{\crl}{{}^{[l]}}
\newcommand{\cald}{\mathcal D_T}
\newcommand{\gc}{{\rm(GC)\,}}
\newcommand{\ddi}{\delta^\diamond}
\newcommand{\intd}{\int_{\R^d}}
\newcommand{\gf}[1]{G_\vp^{j_#1}}
\newcommand{\party}[1]{\frac{\partial^{#1}}{\partial y_{j_1}\cdots\partial y_{j_{#1}}}}

\begin{document}

\title[Stochastic calculus for Gaussian processes]{On Stratonovich and Skorohod stochastic calculus for Gaussian processes}

\date{\today}

\author{Yaozhong Hu \and Maria Jolis \and Samy Tindel}
\date{\today}
\begin{abstract}
In this article, we derive a Stratonovich and Skorohod type change of variables formula for a multidimensional Gaussian process
 with low H\"older regularity $\gamma$ (typically $\ga\le 1/4$). To this aim,
  we combine tools from rough paths theory and stochastic analysis.
\end{abstract}

\address{Yaozhong Hu, Department of Mathematics, University of Kansas, Lawrence, Kansas, 66045 USA.}
\email{hu@math.ku.edu}

\address{Maria Jolis, Departament de Matem\`atiques, Facultat de Ci\`encies, Edifici C, Universitat Aut\`onoma de Barcelona, 08193 Bellaterra, Spain.}
\email{mjolis@mat.uab.cat}

\address{Samy Tindel, Institut \'{E}lie Cartan Nancy, Universit\'e de Nancy 1, B.P. 239,
54506 Vand{\oe}uvre-l\`{e}s-Nancy Cedex, France.}
\email{tindel@iecn.u-nancy.fr}

\thanks{S. Tindel is partially supported by the (French) ANR grant ECRU. M.
Jolis is partially supported by grant MTM2009-08869 Ministerio de
Ciencia e Innovaci\'{o}n and FEDER}

\subjclass[2000]{Primary 60H35; Secondary 60H07, 60H10, 65C30}
\date{\today}
\keywords{Gaussian processes, rough paths, Malliavin calculus, It\^{o}'s formula}

\maketitle

\section{Introduction}
\label{sec:intro}

Starting from the seminal paper \cite{DU}, the stochastic calculus for Gaussian processes has been thoroughly studied during the last decade, fractional Brownian motion being the main example of application of the general results. The literature on the topic includes the case of Volterra processes corresponding to a fBm with Hurst parameter $H> 1/4$ (see \cite{AMN,GRV}), with some extensions to the whole range  $H\in(0,1)$ as in \cite{CH,De,GNRV}. It should be noticed that all those contributions concern the case of real valued  processes,  this feature being an important aspect of the computations.

\smallskip

In a parallel and somewhat different way, the rough path analysis opens the possibility of a pathwise type stochastic calculus for general (including Gaussian) stochastic processes. Let us recall that this theory, initiated by T. Lyons in \cite{Ly1} (see also \cite{FV-bk,LQ-bk,Gu} for  introductions to the topic),
  states that if a $\ga$-H\"{o}lder process $x$ allows to define sufficient number of  iterated integrals   then:
\begin{enumerate}
\item
One gets a Stratonovich type change of variable for $f(x)$ when $f$ is smooth enough.
\item
Differential equations driven by $x$ can be reasonably defined and solved.
\end{enumerate}
In particular, the rough path method is still the only way to solve differential equations driven by Gaussian processes
with H\"older regularity exponent less  than $1/2$,  except for some very particular (e.g. Brownian, linear or one-dimensional) situations.

\smallskip

More specifically, the rough path theory relies on the following set of assumptions:
\begin{hypothesis}\label{hyp:rough-intro}
Let $\ga\in(0,1)$ and $x:\ott\to\R^d$ be a $\ga$-H\"{o}lder process. Consider also the  $n\textsuperscript{th}$ order simplex $\cs_{n,T}=\{ (u_1,\dots, u_n): 0\le u_1 < \cdots < u_n \le T\}$ on $[0,T]$. The process $x$ is supposed to generate a rough path,
 which can be understood as a stack $\{\bx^{\bn};\, n\le \lfloor 1/\ga\rfloor\}$ of functions of two variables satisfying the following three properties:

\smallskip

\noindent
\textit{(1)} \textit{Regularity:}
Each component of $\mathbf{x}^{\mathbf{n}}$ is $n\ga$-H\"{o}lder continuous
(in the sense of the H\"older norm introduced in (\ref{eq:def-norm-C2}))    for all $n\le \lfloor 1/\ga\rfloor$, and $\mathbf{x}^{\mathbf{1}}_{st}=x_t-x_s$.

\smallskip

\noindent
\textit{(2)}
{\it Multiplicativity}:
%%%
Letting $(\der \mathbf{x}^{\mathbf{n}})_{sut}:=\mathbf{x}^{\mathbf{n}}_{st}-\mathbf{x}^{\mathbf{n}}_{su}-\mathbf{x}^{\mathbf{n}}_{ut}$ for $(s,u,t)\in\cs_{3,T}$, one requires
\begin{equation}\label{eq:multiplicativity}
(\der \mathbf{x}^{\mathbf{n}})_{sut}(i_1,\ldots,i_n)=\sum_{n_1=1}^{n-1}
\mathbf{x}_{su}^{\mathbf{n_1}}(i_1,\ldots,i_{n_1}) \mathbf{x}_{ut}^{\mathbf{n-n_1}}(i_{n_1+1},\ldots,i_n).
\end{equation}

\smallskip

\noindent
\textit{(3)}  \textit{Geometricity:}
For any $n,m$ such that $n+m\le \lfloor 1/\ga\rfloor$and $(s,t)\in\cs_{2,T}$, we have:
\beq\label{eq:geom-rough-path}
\bx^{\bn}_{st}(i_1,\ldots,i_n) \, \bx^{\bm}_{st}(j_1,\ldots,j_m)
=\sum_{\bar k\in\mbox{{\tiny Sh}}(\bar\imath,\bar\jmath)}
\bx^{\bn+\bm}_{st}(k_1,\ldots,k_{n+m}),
\eeq
where, for two tuples $\bar\imath,\bar\jmath$, $\Sigma_{(\bar\imath,\bar\jmath)}$ stands for the set of permutations of the indices contained in $(\bar\imath,\bar\jmath)$, and $\mbox{Sh}(\bar\imath,\bar\jmath)$ is a subset of $\Sigma_{(\bar\imath,\bar\jmath)}$ defined by:
$$
\mbox{Sh}(\bar\imath,\bar\jmath)=
\lcl  \si\in \Sigma_{(\bar\imath,\bar\jmath)}; \,
\si \mbox{ does not change the orderings of } \bar\imath \mbox{ and } \bar\jmath \rcl.
$$
\end{hypothesis}
With this set of abstract assumptions in hand, one can define integrals like $\int f(x)\, dx$ in a natural way (as recalled later in the article), and more generally set up the basis of a differential calculus with respect to $x$. Notice that according to T. Lyons terminology~\cite{LQ-bk}, the family $\{\bx^{\bn};\, n\le \lfloor 1/\ga\rfloor\}$ is said to be a weakly geometric rough path above $x$.

\smallskip

Without any surprise, some substantial efforts have been made in the last past years in order to construct rough paths above a wide class of Gaussian processes, among which emerges the case of fractional Brownian motion. Let us recall that a fractional Brownian motion $B$ with Hurst parameter $H\in(0,1)$, defined on a complete probability space $(\oom,\cf,\bp)$, is a $d$-dimensional centered Gaussian process. Its law is thus characterized by its covariance function, which is given by
\begin{equation}\label{eq:cov-fBm}
\be \lc B_t(i)B_s(i) \rc= \frac 12 \lp t^{2H} + s^{2H} - |t-s|^{2H}  \rp \, \1_{(i=j)},
\qquad s,t\in\R_+.
\end{equation}
The variance of the   increments of $B$ is  then given by
$$
\be\lc  \lp  B_t(i)-B_s(i) \rp^2\rc = (t-s)^{2H}, \qquad (s,t)\in\cs_{2,T}, \quad i=1,\ldots, d,
$$
and this implies that  almost surely the trajectories of the fBm are $\gamma$-H\"{o}lder
 continuous for any $\gamma<H$. Furthermore, for $H=1/2$, fBm coincides with the usual Brownian motion,
  converting the family $\{B=B^H;\, H\in(0,1)\}$ into the most natural generalization of this classical process.
  This is why $B$ can be considered as one of the canonical examples of application of the abstract rough path theory.

\smallskip

Until very recently, the rough path constructions for fBm were based on pathwise
type approximations of $B$, as in \cite{CQ,NTU,Un}. Namely, these references all use an
approximation of $B$ by a regularization $B^{\ep}$, consider the associated (Riemann)
iterated integrals $\bb^{\bn,\ep}$ and show their convergence, yielding the existence of a geometric rough path above $B$.
 These approximations all fail for $H\le 1/4$. Indeed, the oscillations of $B$ are then too heavy to define even $\bb^{\2}$
  following this kind of argument, as illustrated by \cite{DNN}. Nevertheless, the article \cite{LV} asserts that a rough path exists above any $\ga$-H\"{o}lder function, and the recent progresses  \cite{NT,Un} show that different concrete rough paths above fBm (and more general processes) can be exhibited, even if those rough paths do not correspond to a regularization of the process at stake.

\smallskip

Summarizing what has been said up to now, there are (at least) two
ways to handle stochastic calculus for Gaussian processes: (i)
Stochastic analysis tools, mainly leading to a Skorohod type
integral (ii) Rough paths analysis, based on the pathwise
convergence of some Riemann sums  and giving rise to a Stratonovich
type integral. Though some efforts  have  been made in \cite{Co} in
order to relate the two approaches (essentially for a fBm with Hurst
parameter $H>1/4$), the current article proposes to delve deeper
into this direction. Namely, we plan to  tackle three different
problems:

\smallskip

\noindent
\textbf{(1)} We show that, starting from a given rough path of order $N$ above a $d$-dimensional process $x$, one can derive a Stratonovich change of variables of the form
\begin{equation}\label{eq:strato-intro}
f(x_t)-f(x_s)= \sum_{i=1}^{d} \int_s^t \partial_{i}f(x_u) \, dx_u(i)
:=\cj_{st}\lp  \nabla f(x_u) \, dx_u \rp,
\end{equation}
for any $f\in C^{N+1}(\R^d;\R)$,
%function $f$ bounded together with
%all its derivatives,
 and where $\partial_{i}f$ stands for $\partial
f/\partial x_i$. This formula is not new, and is in fact an
immediate consequence of the powerful stability theorems which can
be derived from the abstract rough paths theory (see e.g
\cite{FV-bk}). However, we have included these considerations here
for several reasons: \emph{(i)} This paper not being dedicated to
rough paths specialists, we find it useful to include a self
contained, short and simple enough introduction to equation
(\ref{eq:strato-intro}) \emph{(ii)} Our proof is slightly different
from the original one, in the sense that we only rely on the
algebraic and analytic assumptions  of Hypothesis
\ref{hyp:rough-intro} rather than on a limiting procedure
\emph{(iii)} Proving~(\ref{eq:strato-intro}) is also a way for us to
introduce all the objects and structures needed later on for the
Skorohod type calculus. In particular, we derive the following
representation for the integral $\cj_{st}(  \nabla f(x_u) \, dx_u
)$: consider
 a family of partitions $\Pi_{st} = \{s=t_0,\dots,t_n=t\}$ of $[s,t]$, whose mesh tends to 0.
 Then, denoting by $N=\lfloor\frac1{\ga}\rfloor$,
\begin{equation}\label{eq:riemann-sums-intro}
\cj_{st}\lp  \nabla f(x_u) \, dx_u \rp
=\lim_{|\Pi_{st}|\to 0}\sum_{q=0}^{n-1} \sum_{k=0}^{N-1}
\frac{1}{k!}\partial_{i_k\ldots i_1 i}^{k+1}f(x_{t_q}) \,
\bx^{\1}_{t_{q}t_{q+1}}(i_k)\,\cdots\,\bx^{\1}_{t_{q}t_{q+1}}(i_1)\,\bx^{\1}_{t_{q}t_{q+1}}(i).
\end{equation}
These modified Riemann sums will also be essential in the analysis of Skorohod type integrals.

\smallskip

\noindent
\textbf{(2)}
We then specialize our considerations to a Gaussian setting, and use Malliavin calculus tools (in particular some elaborations of~\cite{CH,De}). Namely, supposing that $x$ is a Gaussian process, plus mild additional assumptions on its covariance function, we are able to prove the following assertions: \\
\emph{(i)} Consider a  $C^2(\R^d;\R)$ function $f$ with exponential growth, and $0\le s <t<\infty$. Then the function $u\mapsto \1_{[s,t)}(u)\nabla f(x_u)$ lies into the domain of an extension of the divergence operator (in the Malliavin calculus sense) called $\der^{\di}$. \\
\emph{(ii)} The following Skorohod type formula holds true:
\begin{equation}\label{eq:sko-intro}
f(x_t)-f(x_s)= \ddi\lp \1_{[s,t)}\,\nabla f(x)\rp + \frac12\int_s^t \Delta f(x_u)\,R'_u \,du,
\end{equation}
where $\Delta$ stands for the Laplace operator, $u\mapsto R_u:=\be[|x_u(1)|^2]$ is assumed to be a differentiable function, and $R'$ stands for its derivative.\\
It should be emphasized here that formula (\ref{eq:sko-intro}) is obtained by means of stochastic analysis methods only, independently
of the H\"{o}lder regularity of $x$. Otherwise stated, as in many instances of Gaussian analysis, pathwise regularity can be replaced by a regularity
 on the underlying Wiener space. When both, regularity of the paths
 and on the underlying Wiener space, are satisfied we obtain the relation between the
 Stratonovich type integral and the extended divergence operator.

Let us mention at this point the recent work \cite{KR} that
considers  similar problems as ours. In that article, the authors
define also an extended divergence type operator for 
Gaussian processes (in the one-dimensional case only) with very irregular  covariance and study its
relation with a Stratonovich type integral. For the definition of
the extended divergence, some conditions on the distributional
derivatives of the covariance function $R$ are imposed,  one of them
being that $\partial_{st}^{2} R_{st}$ satisfies
that $\bar\mu(ds,dt):=\partial_{st}^{2} R_{st}\,(t-s)$ (that is well defined) is the difference of two
Radon measures. Our conditions on $R$ are of different nature, we
suppose more regularity but only for the first partial derivative of
$R$ and the variance function. On the other hand, the definition of the Stratonovich type integral in~\cite{KR} is obtained through a regularization approach instead of rough paths theory. As a consequence, some additional regularity conditions on the Gaussian process have to be imposed, while we just rely on the existence of a rough path above $x$.

\smallskip

\noindent
\textbf{(3)}
Finally, one can relate the two stochastic integrals introduced so far by means of modified Wick-Riemann sums. Indeed, we shall show that the integral $\ddi\lp \1_{[s,t)}\,\nabla f(x)\rp$ introduced at relation (\ref{eq:sko-intro}) can also be expressed as
\begin{equation}\label{eq:wick-riemann-intro}
\ddi\lp \1_{[s,t)}\,\nabla f(x)\rp=
\lim_{|\Pi_{st}|\to 0}\sum_{q=0}^{n-1} \sum_{k=0}^{N-1}
\frac{1}{k!}\partial_{i_k\ldots i_1 i}^{k+1}f(x_{t_q}) \di
\bx^{\1}_{t_{q}t_{q+1}}(i_k)\di \cdots\,\bx^{\1}_{t_{q}t_{q+1}}(i_1)\di \bx^{\1}_{t_{q}t_{q+1}}(i),
\end{equation}
where the (almost sure) limit is still taken along a family of
partitions $\Pi_{st} = \{s=t_0,\dots,t_n=t\}$ of $[s,t]$ whose mesh
tends to 0, and where $\di$ stands for the usual Wick product of
Gaussian analysis. This result can be seen as the main contribution
of our paper, and is obtained by a combination of rough paths and
stochastic analysis methods. Specifically, we have mentioned that
the modified Riemann sums in (\ref{eq:riemann-sums-intro}) can be
proved to be convergent by means of rough paths analysis. Our main
additional technical task will thus consist in computing the
correction terms between those Riemann sums and the Wick-Riemann
sums which appear in (\ref{eq:wick-riemann-intro}). This is the aim
of the general Proposition \ref{basic2} on Wick products, which has
an interest in its own right, and is the key ingredient of our
proof. It is worth mentioning at this point that Wick products are
usually introduced within the landmark of white noise analysis. We
rather rely here on the introduction given in \cite{HY}, using the
framework of Gaussian spaces. Let us also mention that Riemann-Wick  sums have been used in \cite{DHP} to study Skorohod
stochastic calculus with respect to (one-dimensional) fBm for $H$ greater than $1/2$, the case of $1/4<H\le 1/2$ being treated in \cite{NTa}. We go beyond these case in Theorem \ref{itosko-d}, and will go back to the link between our formulas and the one produced in \cite{NTa} at Section \ref{sec:comparison}.

\smallskip

In conclusion, this article is devoted to show that
 Stratonovich and Skorohod stochastic calculus are possible
  for a wide range of Gaussian processes. A link between the integrals corresponding
  to those stochastic calculus is made through the introduction of Riemann-Wick modified sums. On the other hand,
   the reader might have noticed that the integrands considered in our stochastic integrals are restricted to processes of the form $\nabla f(x)$. The symmetries of this kind of integrand simplify the analysis of the Stratonovich-Skorohod corrections, reducing all the calculations to corrections involving $\bx^{\1}$ only. An extension to more general integrands would obviously require a lot more in terms of Wick type computations, especially for the terms involving $\bx^{\bk}$ for $k\ge 2$, and is deferred to a subsequent publication.

\smallskip

Here is how our paper is  organized:  Section \ref{sec:one-dim}   recalls  some basic elements of rough paths theory which will be useful in the sequel. Then, as a warmup for the non initiated reader, we derive a Stratonovich change of variable formula in the case of a rough path of order 2 at Section \ref{sec:strato-order-2}. The case of a rough path of arbitrary order is then treated at Section \ref{sec:strato-order-N}. We obtain a Skorohod change of variable with Malliavin calculus tools only at Section \ref{sec:sko-mall}. Finally, the representation of this Skorohod integral by Wick-Riemann sums is performed at Section \ref{sec:rep-sko}.

\section{Some elements of algebraic integration}
\label{sec:one-dim}
As  already mentioned  in the introduction, our stochastic calculus will appeal to the algebraic integration theory, which is a variant of the rough paths theory introduced in~\cite{Gu}, and for which we also refer to \cite{GT} for a detailed introduction.

\subsection{Increments}\label{incr}

The extended pathwise integration we will deal with is based on the
notion of `increments', together with an elementary operator $\der$
acting on them. The algebraic structure they generate is described
in \cite{Gu,GT}, but here we present  directly the definitions of
interest for us, for sake of conciseness. First of all,  for an
arbitrary real number $T>0$, a vector space $V$ and an integer $k\ge
1$ we denote by $\cac_k(V)$ the set of functions $g : [0,T]^{k} \to
V$ such that $g_{t_1 \cdots t_{k}} = 0$ whenever $t_i = t_{i+1}$ for
some $i\le k-1$. Such a function will be called a
\emph{$(k-1)$-increment}, and we   set $\cac_*(V)=\cup_{k\ge
1}\cac_k(V)$. We can now define the announced elementary operator
$\der$ on $\cac_k(V)$:
\begin{equation}
 \label{eq:coboundary}
\delta : \cac_k(V) \to \cac_{k+1}(V), \qquad
(\delta g)_{t_1 \cdots t_{k+1}} = \sum_{i=1}^{k+1} (-1)^{k-i}
g_{t_1  \cdots \hat t_i \cdots t_{k+1}} ,
\end{equation}
where $\hat t_i$ means that this particular argument is omitted.
A fundamental property of $\der$, which is easily verified,
is that
$\delta \delta = 0$, where $\delta \delta$ is considered as an operator
from $\cac_k(V)$ to $\cac_{k+2}(V)$.
We    denote $\cz\cac_k(V) = \cac_k(V) \cap \text{Ker}\delta$
and $\cb \cac_k(V) =
\cac_k(V) \cap \text{Im}\delta$.

\vspace{0.3cm}

Some simple examples of actions of $\der$,
which will be the ones we will really use throughout the paper,
are obtained by letting
$g\in\cac_1$ and $h\in\cac_2$. Then, for any $s,u,t\in\ott$, we have
\begin{equation}
\label{eq:simple_application}
 (\der g)_{st} = g_t - g_s,
\quad\mbox{ and }\quad
(\der h)_{sut} = h_{st}-h_{su}-h_{ut}.
\end{equation}
Furthermore, it is easily checked that
$\cz \cac_{k+1}(V) = \cb \cac_{k}(V)$ for any $k\ge 1$.
In particular, the following basic property holds:
\begin{lemma}\label{exd}
Let $k\ge 1$ and $h\in \cz\cac_{k+1}(V)$. Then there exists a (non unique)
$f\in\cac_{k}(V)$ such that $h=\der f$.
\end{lemma}

\noindent{\it Proof}. This elementary proof is included in
\cite{Gu}, and will be omitted here. However, let us  mention that
$f_{t_1\ldots t_{k}}=(-1)^{k+1}h_{0t_1\ldots t_{k}}$ is a possible
choice. \fin

Observe that Lemma \ref{exd} implies that all the elements
$h \in\cac_2(V)$ such that $\der h= 0$ can be written as $h = \der f$
for some (non unique) $f \in \cac_1(V)$. Thus we get a heuristic
interpretation of $\der |_{\cac_2(V)}$:  it measures how much a
given 1-increment  is far from being an  exact increment of a
function, i.e., a finite difference.

\vspace{0.3cm}

Notice that our future discussions will mainly rely on
$k$-increments with $k \le 2$, for which we will make  some
analytical assumptions. Namely,
we measure the size of these increments by H\"older norms
defined in the following way: for $f \in \cac_2(V)$ let
\begin{equation}\label{eq:def-norm-C2}
\norm{f}_{\mu} =
\sup_{s,t\in\ott}\frac{|f_{st}|}{|t-s|^\mu},
\quad\mbox{and}\quad
\cac_2^\mu(V)=\lcl f \in \cac_2(V);\, \norm{f}_{\mu}<\infty  \rcl.
\end{equation}
Obviously, the usual H\"older spaces $\cac_1^\mu(V)$ will be determined
       in the following way: for a continuous function $g\in\cac_1(V)$, we simply set
\begin{equation}\label{def:hnorm-c1}
\|g\|_{\mu}=\|\der g\|_{\mu},
\end{equation}
and we will say that $g\in\cac_1^\mu(V)$ iff $\|g\|_{\mu}$ is finite.
Notice that $\|\cdot\|_{\mu}$ is only a semi-norm on $\cac_1(V)$,
but we will generally work on spaces of the type
\begin{equation}\label{def:hold-init}
\cac_{1,a}^\mu(V)=
\lcl g:\ott\to V;\, g_0=a,\, \|g\|_{\mu}<\infty \rcl,
\end{equation}
for a given $a\in V$, on which $\|g\|_{\mu}$ defines a distance in
the usual way.
For $h \in \cac_3(V)$ set in the same way
\begin{eqnarray}
 \label{eq:normOCC2}
 \norm{h}_{\gamma,\rho} &=& \sup_{s,u,t\in\ott}
\frac{|h_{sut}|}{|u-s|^\gamma |t-u|^\rho}\\
%\quad\mbox{and}\quad
\|h\|_\mu &= &
\inf\left \{\sum_i \|h_i\|_{\rho_i,\mu-\rho_i} ;\, h =
\sum_i h_i,\, 0 < \rho_i < \mu \right\} ,\nonumber
\end{eqnarray}
where the last infimum is taken over all sequences $\{h_i \in \cac_3(V) \}$
such that $h
= \sum_i h_i$ and for all choices of the numbers $\rho_i \in (0,z)$.
Then  $\|\cdot\|_\mu$ is easily seen to be a norm on $\cac_3(V)$, and we set
$$
\cac_3^\mu(V):=\lcl h\in\cac_3(V);\, \|h\|_\mu<\infty \rcl.
$$
Eventually,
let $\cac_3^{1+}(V) = \cup_{\mu > 1} \cac_3^\mu(V)$,
and notice  that the same kind of norms can be considered on the
spaces $\cz \cac_3(V)$, leading to the definition of some spaces
$\cz \cac_3^\mu(V)$ and $\cz \cac_3^{1+}(V)$.

\vspace{0.3cm}

With these notations in mind
the following proposition is a basic result, which  belongs to  the core of
our approach to pathwise integration. Its proof may be found
in a simple form in \cite{GT}.
\begin{proposition}[The $\Lambda$-map]
\label{prop:Lambda}
There exists a unique linear map $\Lambda: \cz \cac^{1+}_3(V)
\to \cac_2^{1+}(V)$ such that
$$
\delta \Lambda  = \id_{\cz \cac_3^{1+}(V)}
\quad \mbox{ and } \quad \quad
\Lambda  \delta= \id_{\cac_2^{1+}(V)}.
$$
In other words, for any $h\in\cac^{1+}_3(V)$ such that $\der h=0$
there exists a unique $g=\laa(h)\in\cac_2^{1+}(V)$ such that $\der g=h$.
Furthermore, for any $\mu > 1$,
the map $\laa$ is continuous from $\cz \cac^{\mu}_3(V)$
to $\cac_2^{\mu}(V)$ and we have
\begin{equation}\label{ineqla}
\|\Lambda h\|_{\mu} \le \frac{1}{2^\mu-2} \|h\|_{\mu} ,\qquad h \in
\cz \cac^{\mu}_3(V).
\end{equation}
\end{proposition}

\smallskip

Let us mention at this point a first link between the structures
we have introduced so far and the problem of integration of irregular
functions.

\begin{corollary}
\label{cor:integration}
For any 1-increment $g\in\cac_2 (V)$ such that $\der g\in\cac_3^{1+}$,
set
$
\delta f = (\id-\Lambda \delta) g
$.
Then
$$
(\delta f)_{st} = \lim_{|\Pi_{st}| \to 0} \sum_{i=0}^{n-1}
g_{t_i\,t_{i+1}},
$$
where the limit is over any partition $\Pi_{st} = \{t_0=s,\dots,
t_n=t\}$ of $[s,t]$, whose mesh tends to zero. Thus, the 1-increment
$\delta f$ is the indefinite integral of the 1-increment $g$.
\end{corollary}

\begin{proof}
Just consider the equation $g = \delta f + \Lambda \delta g$ and write
\begin{equation*}
 \begin{split}
S_{\Pi_{st}}& = \sum_{i=0}^{n-1} g_{t_i\, t_{i+1}}
%\\ &
= \sum_{i=0}^{n-1} (\delta f)_{t_i\, t_{i+1}} + \sum_{i=0}^{n-1}
(\Lambda \delta g)_{t_i\, t_{i+1}}
\\ & =
(\delta f)_{st} + \sum_{i=0}^{n-1} (\Lambda \delta g)_{t_i\,
t_{i+1}}.
     \end{split}
\end{equation*}
Then observe that, due to the fact that $\Lambda \delta g \in
\cac_2^{1+}(V)$, the last sum converges to zero.

\end{proof}

\subsection{Computations in $\cac_*$}\label{cpss}

Let us specialize now to the case $V=\R$, and just write $\cac_{k}^{\ga}$ for $\cac_{k}^{\ga}(\R)$. Then $(\cac_*,\delta)$ can be endowed with the following product:
for  $g\in\cac_n$ and $h\in\cac_m$ let  $gh$
be the element of $\cac_{n+m-1}$ defined by
\begin{equation}\label{cvpdt}
(gh)_{t_1,\dots,t_{m+n+1}}=
g_{t_1,\dots,t_{n}} h_{t_{n},\dots,t_{m+n-1}},
\quad
t_1,\dots,t_{m+n-1}\in\ott.
\end{equation}
In this context, we have the following useful properties.

\begin{proposition}\label{difrul}
The following differentiation rules hold true:
\begin{enumerate}
\item
Let $g\in\cac_1$ and $h\in\cac_1$. Then
$gh\in\cac_1$ and
\begin{equation}\label{difrulu}
\der (gh) = \der g\,  h + g\, \der h.
\end{equation}
\item
Let $g\in\cac_1$ and $h\in\cac_2$. Then
$gh\in\cac_2$ and
\begin{equation}
\der (gh) = \der g\, h - g \,\der h.
\end{equation}
\item
Let $g\in\cac_2$ and $h\in\cac_1$. Then
$gh\in\cac_2$ and
\begin{equation}
\der (gh) = \der g\, h  + g \,\der h.
\end{equation}
\end{enumerate}
\end{proposition}

\begin{proof}
We will just prove (\ref{difrulu}), the other relations being just as  simple.
If $g,h\in\cac_1 $, then
$$
\lc \der (gh) \rc_{st}
= g_th_t-g_sh_s
=g_s\lp h_t-h_s \rp +\lp  g_t-g_s\rp h_t\\
=g_s \lp \der h \rp_{st}+ \lp \der g \rp_{st} h_t,
$$
which proves our claim.

\end{proof}

\smallskip

The iterated integrals of smooth functions on $\ott$ are %obviously
particular cases of elements of $\cac_2$, which will be of interest

for us. Let us recall  some basic  rules for these objects: consider
$f\in\cac_1^\infty$ and $g\in\cac_1^\infty$, where $\cac_1^\infty $
denotes the set of smooth functions on $\ott$. Then the integral
$\int f \, dg$, which will be denoted  indistinctly by $\int f\, dg$
or $\cj(f\, dg)$, can be considered as an element of
$\cac_2^\infty$. Namely, for $(s,t)\in\cs_{2,T}$ we set
$$
\cj_{st}(f\,  dg)
=
\left(\int  f dg \right)_{st} = \int_s^t   f_u dg_u.
$$
The multiple integrals can also be defined in the following way:
given a smooth element $h \in \cac_2^\infty$ and $(s,t)\in\cs_{2,T}$, we set
$$
\cj_{st}(h\, dg )\equiv
\left(\int h dg  \right)_{st} = \int_s^t  h_{su} dg_u .
$$
In particular,  for
$f^1\in\cac_1^\infty$, $f^2\in\cac_1^\infty$
and $f^3\in\cac_1^\infty$ the double integral
$\cj_{st}( f^3\, df^2 df^1)$ is defined  as
$$
\cj_{st}( f^3\, df^2df^1)
=\lp \int f^3\, df^2 df^1  \rp_{st}
= \int_s^t \cj_{su}\lp f^3\, df^2  \rp \, df_u^1.
$$
Now suppose that the $n$th order iterated integral of
$f^{n+1}df^n\cdots df^2$, which is  denoted by
$\cj(f^{n+1}df^n\cdots df^2)$, has been defined for
$f^j\in\cac_1^\infty$.
Then, if $f^1\in\cac_1^\infty$, we set
\begin{equation}\label{multintg}
\cj_{st}(f^{n+1}df^n \cdots df^2 df^1)
=
\int_s^t  \cj_{su}\lp f^{n+1}df^n\cdots df^2\rp  \, df_u^{1},
\end{equation}
which recursively defines the iterated integrals of smooth
functions. Observe that an   $n$th order integral $\cj(df^n\cdots df^2
df^1)$ can be defined along the same lines, starting with
$$\cj(df)=\delta f,$$
$$\cj_{st}(df^2\,df^1)=\int_s^t \cj_{su}(df^2)\,df^1_u=\int_{s}^t \big( \delta
f^2\big)_{su}\,df^1_u,$$ and so on.

\medskip

The following relations between multiple integrals and the operator $\der$ will also be useful. The reader is sent to \cite{GT} for its elementary proof.
\begin{proposition}\label{dissec}
Let $f\in\cac_1^\infty$ and $g\in\cac_1^\infty$.
Then it holds that
$$
\der g = \cj( dg), \qquad \der\lp \cj(f dg)\rp = 0, \qquad \der\lp
\cj (df dg)\rp = (\der f) (\der g) = \cj(df) \cj(dg),
$$
and
$$
\der \lp \cj( df^n \cdots df^1)\rp  =
\sum_{i=1}^{n-1}
\cj\lp df^n \cdots df^{i+1}\rp \cj\lp df^{i}\cdots df^1\rp.
$$
\end{proposition}

\section{Stratonovich calculus of order 2}
\label{sec:strato-order-2} This section is devoted to establish an
It\^{o}-Stratonovich change of variable formula for a process
$x\in\cac_1^{\ga}(\R^d)$, with $1/3<\ga\le 1/2$, provided this
process generates a (weakly geometric) L\'{e}vy area. It is intended as a warm up for the general change of variable of the next section, especially for those readers who might not be acquainted to rough paths techniques.

\subsection{Weakly controlled processes}\label{sec:wc-ps}

Recall that we have in mind to give a change of variable formula for
$f(x)$ when $x$ is a function in $\cac_1^{\ga}(\R^d)$ with $\ga>1/3$
and $f$ is a sufficiently smooth   function. In this case, the rough path
above $x$ is reduced to a second order iterated integral, and the
multiplicative property (\ref{eq:multiplicativity}) of the path can
be read as:
\begin{hypothesis}\label{hyp:x}
The path $x$ is $\R^d$-valued $\ga$-H\"older with $\ga>1/3$ and
admits a L\'evy area, that is a process
$\xd\in\cac_2^{2\ga}(\R^{d,d})$ satisfying
$$\,
\der\xd=\bx^{\1}\otimes \bx^{\1},
\quad\mbox{i.\!\! e.}\quad
\lc (\der\xd)_{sut} \rc(i,j)
=
[\bx^{\1}(i)]_{su} [\bx^{\1}(j)]_{ut},
$$
for $s,u,t\in\cs_{3,T}$ and $i,j\in\{1,\ldots,d  \}$.
\end{hypothesis}
\noindent
%The integral $\int f(x_u) \,dx_u$ will then be expressed as a continuous function of the input $\bx^{\1}$ and $\xd$.

\smallskip

Let us now be  more specific about the global strategy we will adopt in order
to obtain our Stratonovich type formula.
First of all,  we shall define integrals with respect to $x$ for a class of integrands
called weakly controlled processes,
 that we proceed to define. Notice that in the following definition we   use for the
 first time the convention of summation
 over repeated indices, which will prevail until the end of Section \ref{sec:strato-order-N}.
\begin{definition}\label{def39}
Let $z$ be a process in $\cac_1^\ga(\R^n)$ with $1/3<\ga\le 1/2$
(that is, $N:=\lfloor 1/\ga\rfloor=2$). We say that $z$ is a weakly controlled path based
on $x$ and starting from $a$ if $z_0=a$, which is a given initial condition in $\R^n$,
and $\der z\in\cac_2^\gamma(\R^n)$ can be decomposed into
\begin{equation}\label{weak:dcp}
\der z(i)=\zeta(i,i_1) \bx^{\1}(i_1)+ r(i),
\quad\mbox{i.\!\! e.}\quad
(\der z(i))_{st}=\zeta_s(i,i_1) \bx^{\1}_{st}(i_1) + r_{st}(i),
\end{equation}
for all $(s,t)\in\cs_{2,T}$. In the previous formula, we assume
$\zeta\in\cac_1^\ga(\R^{n,d})$, and $r$ is a regular part
such that $r\in\cac_2^{2\ga}(\R^n)$. The space of weakly controlled
paths starting from $a$ will be denoted by $\cq_{\ga,a}(\R^n)$, and a process
$z\in\cq_{\ga,a}(\R^n)$ can be considered in fact as a couple
$(z,\zeta)$. The natural semi-norm on $\cq_{\ka,a}(\R^k)$ is given
by
$$
\cn[z;\cq_{\ga,a}(\R^n)]=
\cn[z;\cac_1^{\ga}(\R^n)]
+ \cn[\zeta;\cac_1^{\infty}(\R^{n,d})]
+ \cn[\zeta;\cac_1^{\ga}(\R^{n,d})]
+\cn[r;\cac_2^{2\ga}(\R^n)],
$$
with $\cn[g;\cac_1^{\ka}]$ defined by
(\ref{def:hnorm-c1}) and
$\cn[\zeta;\cac_1^{\infty}(V)]=\sup_{0\le s\le T}|\zeta_s|_V$.
\end{definition}

With this definition at  hand, we will try to obtain our change of variables formula in the following way:
\begin{enumerate}
\item
Study the decomposition of $f(x)$ as weakly controlled process, when $f$ is a smooth function.
\item
Define rigorously the integral $\int z_u dx_u=\cj(z dx)$
for a weakly controlled path $z$ and compute its decomposition
(\ref{weak:dcp}).
\item
Compare the decompositions of $f(x)$ and $\int \nabla f(x)\, dx$,
and show that they coincide, up to a term with H\"{o}lder regularity
greater than 1.
\end{enumerate}
In this section, we will concentrate on the first point of the program.

\smallskip

Let us see then how to decompose $f(x)$ as a controlled process when $f$ is a smooth enough function, for which we first introduce a convention which will hold true until the end of the paper: for any smooth function $f:\R^d\to\R$, $k\ge 1$, $(i_1,\ldots,i_k)\in\{1,\ldots, d\}^k$ and $\xi\in\R^d$, we set
\begin{equation}\label{eq:convention-partial-deriv}
\partial_{i_1\ldots i_k}f(\xi)^{k}=\frac{\partial^k f}{\partial x_{i_1}\cdots \partial x_{i_k}}(\xi).
\end{equation}
With this notation in hand, our decomposition result is the following:
\begin{proposition}\label{cp:weak-phi}
Let $f:\R^d\to\R$ be a $C_b^2$ function such that
$f(x_0)=a$, and set $z=f(x)$. Then $z\in\cq_{\ga,a}$,  and it can be decomposed into
$
\der z=\zeta \der x +r,
$
with
$$
\zeta(i)= \partial_{i}f(x)
\quad\mbox{ and }\quad
r= \der f(x)- \partial_{i}f(x) \, \bx^{\1}(i).
$$
Furthermore,
\begin{equation}\label{bnd:phi}
\cn[z;\cq_{\gamma,a}]\le c_{f,T}\lp 1+\cn^2[x;\cac_1^{\ga}(\R^d)]
\rp.
\end{equation}
\end{proposition}

\begin{proof}
The algebraic part of the assertion is straightforward. Just write
\begin{equation}\label{eq:dcp-f-x-1}
(\der z)_{st}=f(x_t)-f(x_s)=
\partial_{i}f(x_s) \, \bx^{\1}_{st}(i) + r_{st},
\end{equation}
which is the desired decomposition.

\vspace{0.3cm}

In order to give an estimate for $\cn[z;\cq_{\ga,a}(\R^n)]$, one has
of course to establish bounds for $\cn[z;\cac_1^{\ga}(\R^n)]$,
$\cn[\zeta;$ $\cac_1^{\ga}(\R^d)]$,
$\cn[\zeta;\cac_1^{\infty}(\R^d)]$ and $\cn[r;\cac_2^{2\ga}]$. These
estimates are readily obtained from decomposition
(\ref{eq:dcp-f-x-1}), and details are left to the reader.

\end{proof}
\begin{remark}\label{finitud}
The algebraic part of the above proposition remains true if we only
suppose that $f\in \cac^2(\R^d)$. Indeed, since $f$ together with its first and
second order partial derivatives and  $x$  are continuous functions on a
compact set, we have that $\ze(i)=\partial_i f(x)\in \cac_1^\ga$ and
$r=\delta f(x)-\partial_if(x)\bx^{\1}(i)\in \cac_2^\ga$.
Nevertheless, the inequality norm (\ref{bnd:phi}) fails and
$\cn[z;\cq_{\gamma,a}]$ cannot be bounded  by terms only depending
on  the H\"{o}lder norm of $x$.
\end{remark}

\subsection{Integration of weakly controlled paths}
Let us now turn  to the integration of weakly controlled paths,
which is summarized in the following theorem.
\begin{theorem}\label{intg:mdx}
For a given $1/3<\ga\le 1/2$, let $x$ be a process satisfying
Hypothesis \ref{hyp:x}. Furthermore,  let $m\in\cq_{\ga,b}(\R^{d})$
with decomposition $m_0=b\in\R^d$ and
\begin{equation}\label{dcp:m}
\der m(i)=\mu(i,i_1) \, \bx^{\1}(i_1)+ r(i),
\quad\mbox{ where }\quad
\mu\in\cac_1^\ga(\R^{d, d}), \, r\in\cac_2^{2\ga}(\R^{d}).
\end{equation}
Define $z$ by $z_0=a\in\R$ and
\begin{equation}\label{dcp:mdx}
\der z=
m(i) \, \bx^{\1}(i) + \mu(i,i_1) \, \bx^{\2}(i_1,i)
-\laa\lp  r(i) \, \bx^{\1}(i) + \der\mu(i,i_1) \, \bx^{\2}(i_1,i) \rp.
\end{equation}
Finally, set
$$
\cj_{st}(m\, dx)=\ist \left\langle m_u, \, dx_u\right\rangle_{\R^d} \triangleq (\der z)_{st}.
$$
Then:

\smallskip

\noindent \emph{(1)} $z$ is well-defined as an element of
$\cq_{\ga,a}(\R)$, and coincides with the Riemann-Stieltjes integral
of $z$ with respect to $x$ whenever these two functions are smooth.

\smallskip

\noindent \emph{(2)} The semi-norm of $z$ in $\cq_{\ga,a}(\R)$ can
be estimated as
\begin{equation}\label{bnd:norm-imdx}
\cn[z;\cq_{\ga,a}(\R)]\le c_{x} \lp 1 +
\cn[m;\cq_{\ga,b}(\R^{d})]\rp,
\end{equation}
for a positive constant $c_{x}$ which can be bounded as follows:
$$
c_x\le c
\lp
\cn[\bx^{\1};\, \cac_2^{\ga}(\R^d)]+ \cn[\bx^{\2};\, \cac_2^{2\ga}(\R^{d^2})]
\rp,
\quad\mbox{ for a universal constant }c.
$$

\smallskip

\noindent
\emph{(3)}
It holds
\begin{equation}\label{rsums:imdx}
\cj_{st}(m\, dx) =\lim_{|\Pi_{st}|\to 0}\sum_{q=0}^{n-1} \lc
m_{t_{q}}(i) \, \bx^{\1}_{t_{q}, t_{q+1}}(i) + \mu_{t_{q}}(i,i_1) \,
\xd_{t_{q}, t_{q+1}}(i_1,i) \rc,
\end{equation} for any $0\le s<t\le T$,
where the limit is taken over all partitions
$\Pi_{st} = \{s=t_0,\dots,t_n=t\}$
of $[s,t]$, as the mesh of the partition goes to zero.
\end{theorem}

\vspace{0.3cm}

Before going into the technical details of the proof,  let us see
how to recover (\ref{dcp:mdx}) in the smooth case, in order to
justify our  definition of the integral. (Notice however that
(\ref{rsums:imdx}) corresponds to the usual definition in the rough paths theory
\cite{LQ-bk}, which gives another kind of justification.)
\\ Let us assume for the moment
that $x$ is a smooth function and that $m\in\cac_1^\infty(\R^{d})$
admits the decomposition (\ref{dcp:m}) with
$\mu\in\cac_1^\infty(\R^{d,d})$ and $r\in\cac_2^{\infty}(\R^{d})$.
Then $\ist\langle m_u, \, dx_u\rangle$ is well-defined, and we have
$$
\ist \left\langle m_u, \, dx_u\right\rangle_{\R^d} =
m_s(i) \, [x_t(i)-x_s(i)] + \ist[m_u(i)-m_s(i)] \, dx_u(i)
$$  for $s\le t$,
respectively
$$
\cj( m \,dx) = m(i) \,\bx^{\1}(i) + \cj\lp\der m(i) \, dx(i)\rp.
$$
Let us now plug the  decomposition (\ref{dcp:m}) into this expression,
which yields
\begin{eqnarray}\label{eq:imdx}
\cj_{st}(m \,dx)&=& m_{s}(i) \, \bx_{st}^{\1}(i)
+\ist \mu_{s}(i,i_1) \, \bx_{su}^{\1}(i_1) \, dx_u(i)
+ \cj(r \,dx)\nonumber\\
&=& m_{s}(i) \, \bx_{st}^{\1}(i) +\mu_{s}(i,i_1) \, \bx^{\2}_{st}(i_1,i) + \cj(r \,dx).
\end{eqnarray}
Notice that the terms
$m\,\der x$ and $\mu\, \xd$  in (\ref{eq:imdx})  are well-defined as soon as $x$ and $\xd$
are defined themselves. In order to push forward our analysis to the
rough case, it remains to handle the term $\cj(r \,dx)$.
Thanks to (\ref{eq:imdx}) we can write
$$
\cj(r \,dx)= \cj(m \,dx)- m(i) \,\bx^{\1}(i) -\mu(i,i_1) \, \bx^{\2}(i_1,i),
$$
and let us analyze this relation by applying $\der$ to both sides.
Using the second part of Proposition \ref{difrul} and the   Proposition \ref{dissec}
yields
\begin{eqnarray}\label{eq:di-rdx}
\der\lc\cj(r \,dx)\rc&=&
- \der\lc m(i) \,\bx^{\1}(i) \rc -\der\lc \mu(i,i_1) \, \bx^{\2}(i_1,i) \rc\nonumber\\
&=& -\der m(i) \, \bx^{\1}(i) -\der\mu(i,i_1) \, \bx^{\2}(i_1,i)
+\mu(i,i_1)\, \bx^{\1}(i_1) \, \bx^{\1}(i) \nonumber\\
&=&-\lc \mu(i,i_1) \,\bx^{\1}(i_1) + r(i) \rc \bx^{\1}(i)
-\der\mu(i,i_1) \, \bx^{\2}(i_1,i)
+\mu(i,i_1)\, \bx^{\1}(i_1) \, \bx^{\1}(i) \nonumber\\
&=& -\der\mu(i,i_1) \, \bx^{\2}(i_1,i) - r(i)\, \bx^{\1}(i).
\end{eqnarray}
Assuming now that the increments $\der\mu(i,i_1) \, \bx^{\2}(i_1,i)$ and $r(i)\, \bx^{\1}(i)$ are all  elements of $\cac_2^\mu$ with $\mu>1$, $\der\mu(i,i_1) \, \bx^{\2}(i_1,i) + r(i)\, \bx^{\1}(i)$ becomes
an element of $\dom(\laa)$, and inserting~(\ref{eq:di-rdx}) into~(\ref{eq:imdx}) we obtain
$$
\der z=\cj(m\, dx)\equiv
m(i) \, \bx^{\1}(i) + \mu(i,i_1) \, \bx^{\2}(i_1,i)
-\laa\lp  r(i) \, \bx^{\1}(i) + \der\mu(i,i_1) \, \bx^{\2}(i_1,i) \rp,
$$
which is the expression (\ref{dcp:mdx}) of our Theorem
\ref{intg:mdx}. Thus (\ref{dcp:mdx}) is a natural expression for
$\cj(m\,  dx)$.

\smallskip

\begin{proof}[Proof of Theorem \ref{intg:mdx}]
We will divide  this proof into  two steps.

\vspace{0.3cm}

\noindent {\it Step 1:} Recalling the assumption $3\ga>1$, let us
analyze the three terms in the right hand side of (\ref{dcp:mdx})
and show that they define an element of $\cq_{\ga,a}$ such that
$\der z =\zeta(i)\, \bx^{\1}(i) +\hat r$ with
$$
\zeta(i)=m(i) \quad\mbox{ and }\quad \hat r= \mu(i,i_1) \,
\bx^{\2}(i_1,i) -\laa\lp  r(i) \, \bx^{\1}(i) + \der\mu(i,i_1) \,
\bx^{\2}(i_1,i) \rp.
$$
Indeed, on   one hand $m\in\cac_1^\ga(\R^d)$ and thus $\ze=m$ is
of the desired form for an element of $\cq_{\ga,a}$. On the other
hand, if $m\in\cq_{\ga,b}$, $\mu$ is assumed to be bounded and since
$\xd\in\cac_2^{2\ga}(\R^{d,d})$ we get that $\mu(i,i_1) \,
\bx^{\2}(i_1,i)\in\cac_2^{2\ga}$. Along the same lines we can prove
that $r(i) \, \bx^{\1}(i)\in \cac_3^{3\ga}$ and $\der\mu(i,i_1) \,
\bx^{\2}(i_1,i)\in \cac_3^{3\ga}$. Since $3\ga>1$, we obtain  that
$r(i) \, \bx^{\1}(i) + \der\mu(i,i_1) \,
\bx^{\2}(i_1,i)\in\dom(\laa)$ and
$$
\laa\lp r\,\der x+\der \mu\,\xd \rp\in \cac_2^{3\ga}.
$$
Thus we have proved that
$$
\hat r= \mu(i,i_1) \, \bx^{\2}(i_1,i) -\laa\lp  r(i) \, \bx^{\1}(i)
+ \der\mu(i,i_1) \, \bx^{\2}(i_1,i) \rp \in \cac_2^{2\ga}
$$
and hence that $z\in\cq_{\ga,a}(\R)$. The estimate
(\ref{bnd:norm-imdx}) is now obtained using
the same kind of considerations and are left to the reader for the
sake of conciseness.

\smallskip

\noindent
{\it Step 2:}
The same kind of computations as those leading to (\ref{eq:di-rdx})
also show that
$$
\der\lp m(i) \, \bx^{\1}(i) + \mu(i,i_1) \, \bx^{\2}(i_1,i)  \rp
=
-\lc r(i) \, \bx^{\1}(i) + \der\mu(i,i_1) \, \bx^{\2}(i_1,i) \rc.
$$
Hence equation (\ref{dcp:mdx}) can also be read as
$$
\cj(m\, dx)=\lc \id-\laa\der \rc \lp m(i) \, \bx^{\1}(i) + \mu(i,i_1) \, \bx^{\2}(i_1,i) \rp,
$$
and a direct application of Corollary \ref{cor:integration} yields
(\ref{rsums:imdx}), which ends our proof.

\end{proof}

\subsection{It\^{o}-Stratonovich formula}
\label{sec:stoch-rough}

We are now ready to obtain a change of variable formula for $f(x)$,
according to the strategy given in Section \ref{sec:wc-ps}. For
this, we need to assume, on top of the multiplicative Hypothesis
\ref{hyp:x}, the following geometric rule which  is
(\ref{eq:geom-rough-path}) in the case
$N=\lfloor\frac1{\ga}\rfloor=2$:
\begin{hypothesis}\label{hyp:sym-xd}
Let $\xd$ be the area process defined in Hypothesis \ref{hyp:x}. Then we assume that, for all $(s,t)\in\cs_{2,T}$, we have
\begin{equation*}
\bx^{\1}_{st}(i) \, \bx^{\1}_{st}(j)=\bx^{\2}_{st}(i,j)+ \bx^{\2}_{st}(j,i).
\end{equation*}
\end{hypothesis}

\smallskip

With these assumptions in mind, our change of variable formula  reads
as follows:
\begin{proposition}\label{prop:ito-strat-form-2}
Assume that $x$ satisfies Hypothesis \ref{hyp:x} and
\ref{hyp:sym-xd}. Let $f$ be  a $C^3(\R^{d};\R)$ function. Then
\begin{equation}\label{itoform}
\lc \der(f(x)) \rc_{st}= \cj_{st}\lp \nabla f(x) \, dx \rp
=\ist \lla  \nabla f(x_u), \, dx_u \rra_{\R^d},
\end{equation}
where the integral above has to be understood in the sense of
Theorem \ref{intg:mdx}.
\end{proposition}

\begin{proof}

Consider a partition $\Pi_{st}=\{s=t_0<\cdots<t_n=t\}$ of $[s,t]$.
We have that
\begin{align}\label{eq-taylor}
&f(x_t)-f(x_s)=
\sum_{q=0}^{n-1}f(x_{t_{q+1}})-f(x_{t_q}) \notag\\
&=\sumq\partial_{i}f(x_{t_q})\,\bx^{\1}_{t_qt_{q+1}}(i)+\frac12\sumq\sum_{i_1,i_2=1}^d\partial^2_{i_1 i_2}
f(x_{t_q})\,\bx^{\1}_{t_qt_{q+1}}(i_1)\,\bx^{\1}_{t_qt_{q+1}}(i_2)\notag\\
&\hspace{3cm}+\frac{1}{3!}\sumq\sum_{i_1,i_2,i_3=1}^d\partial^3_{i_1 i_2 i_3}
f(x_{\xi^q_{i_1i_2i_3}})\,\bx^{\1}_{t_qt_{q+1}}(i_1)\,\bx^{\1}_{t_qt_{q+1}}(i_2)\,\bx^{\1}_{t_qt_{q+1}}(i_3),
\end{align}
for a certain element $\xi^q_{i_1i_2i_3}\in[t_q,t_{q+1}]$. Invoking now
Hypothesis \ref{hyp:sym-xd} and Schwarz rule, one can express the
sum $\frac12\sum_{i_1,i_2=1}^d\partial^2_{i_1 i_2}
f(x_{t_q})\,\bx^{\1}_{t_qt_{q+1}}(i_1)\,\bx^{\1}_{t_qt_{q+1}}(i_2)$ as
$$\frac12\sum_{i_1,i_2=1}^d\partial^2_{i_1 i_2}
f(x_{t_q})\,[\bx^{\2}_{t_qt_{q+1}}(i_1,i_2)+\bx^{\2}_{t_qt_{q+1}}(i_2,i_1)]=\sum_{i_1,i_2=1}^d\partial^2_{i_1 i_2}
f(x_{t_q})\,\bx^{\2}_{t_qt_{q+1}}(i_1,i_2).
$$

So, going back to our convention on repeated indices, one can thus
recast expression~(\ref{eq-taylor}) into
\begin{equation}\label{eq-taylor-2}
f(x_t)-f(x_s)=\sumq\partial_{i}f(x_{t_q})\,\bx^{\1}_{t_qt_{q+1}}(i)
+\sumq\partial^2_{i_1 i_2} f(x_{t_q})\,\bx^{\2}_{t_qt_{q+1}}(i_1,i_2)
+R_{st},
\end{equation}
where $R_{st}$ can be written as $R_{st}=\frac{1}{3!}\sumq \rho_{t_qt_{q+1}}$ with
\begin{equation*}
\rho_{t_qt_{q+1}}=
\partial^3_{i_1 i_2 i_3}
f(x_{\xi^q_{i_1i_2i_3}})\,\bx^{\1}_{t_qt_{q+1}}(i_1)\,\bx^{\1}_{t_qt_{q+1}}(i_2)\,\bx^{\1}_{t_qt_{q+1}}(i_3).
\end{equation*}
Furthermore, it is readily checked that for any $0\le q\le n-1$ we
have $\rho_{t_qt_{q+1}}\le C |t_{q+1}-t_{q}|^{3\ga}$, with $C$ a
constant depending on $f$ and $x$, owing to the fact that $f$ is a
$C^3$ function and $x$ is continuous on $[0,T]$. Thus, since
$3\ga>1$, it is easily seen that $\lim_{|\Pi_{st}|\to 0} \sumq
\rho_{t_qt_{q+1}}=0$. Plugging this relation into
(\ref{eq-taylor-2}), we have proved that
\begin{equation}\label{eq-taylor-3}
f(x_t)-f(x_s)= \lim_{|\Pi_{st}|\to 0}
\sumq \partial_{i}f(x_{t_q})\,\bx^{\1}_{t_qt_{q+1}}(i)
+\sumq\partial^2_{i_1 i_2} f(x_{t_q})\,\bx^{\2}_{t_qt_{q+1}}(i_1,i_2).
\end{equation}

\smallskip

On the other hand, Proposition \ref{cp:weak-phi} asserts that the decomposition of $\nabla f(x)$ as a weakly controlled path is given by
$$\lp\delta\nabla f(x)(i)\rp_{st}=\lp\delta\partial_i
\,f(x)\rp_{st}=\partial_{i,i_1}^2f(x_s)\bx_{st}^{\1}(i_1)\,+\,r_{st}(i),
$$
where $r$ lies into $\cac_2^{2\ga}$
Hence, using formula (\ref{rsums:imdx}), we have that
\begin{equation}\label{eq:intg-riemann-sums}
\cj_{st}(\nabla f(x)dx)=\lim_{|\Pi_{st}|\to 0}\lc \sumq\partial_{i}f(x_{t_q})\,\bx^{\1}_{t_qt_{q+1}}(i)+\sumq\partial^2_{i_1 i_2}
f(x_{t_q})\,\bx^{\2}_{t_qt_{q+1}}(i_1,i_2)\rc.
\end{equation}
Comparing this equality with (\ref{eq-taylor-3}), one
obtain easily the desired It\^{o}-Stratonovich formula.

\end{proof}

\begin{remark}\label{itoconseq}
 The following formula:
\begin{multline}\label{eq:intg-riemann-sums-products}
\cj_{st}(\nabla f(x)dx)\\
=\lim_{|\Pi_{st}|\to 0}
\lc\sumq\partial_{i}f(x_{t_q})\,\bx^{\1}_{t_qt_{q+1}}(i)+\frac12\sumq\sum_{i_1,i_2=1}^d\partial^2_{i_1
i_2}
f(x_{t_q})\,\bx^{\1}_{t_qt_{q+1}}(i_1)\,\bx^{\1}_{t_qt_{q+1}}(i_2)\rc
\end{multline}
is also an interesting byproduct of the proof of
Proposition~\ref{prop:ito-strat-form-2}.
%One should have in mind that formula (\ref{eq:intg-riemann-sums}) is also an interesting byproduct of the proof of Proposition~\ref{prop:ito-strat-form-2}.
\end{remark}

\section{General Stratonovich calculus}
\label{sec:strato-order-N}

We will now handle the case of a weakly geometric rough path based on $x\in\cac_1^{\ga}(\R^d)$ as defined in the introduction, and we set $N=\lfloor 1/\ga \rfloor$. We shall define an integration theory and show an It\^{o}-Stratonovich formula for this kind of process. This being done along the same lines as in Section \ref{sec:strato-order-2}, we may skip some details of computations here. In any case, recall that we suppose the existence of a family $\{\bx^{\bn};\, n\le N\}$ of increments in $\cac_2$ satisfying the regularity, multiplicative and geometric properties given at Section \ref{sec:intro}.

\subsection{Weakly controlled processes}
\label{sec:wc-ps-N}

With respect to the case of order 2, the notion of controlled process is obviously obtained here by introducing more iterated integrals of the process $x$. A new kind of cascade relation is also required, which is reminiscent of the Heisenberg type structure of \cite{LQ-bk}.
\begin{definition}\label{def:weak-ps-N}
Let $z$ be a process in $\cac_1^\ga(\R^n)$ with $1/(N+1)<\ga\le
1/N$. We say that $z$ is a weakly controlled path based on $x$ and starting at $a\in\R^n$, if
$z_0=a$ and $\der
z\in\cac_2^\ga(\R^n)$ can be decomposed into
\begin{equation}\label{eq:weak-dcp-N}
\der z(i)=\sum_{k=1}^{N-1} \zeta^{k}(i,i_1,\ldots,i_k) \, \bx^{\bk}(i_k,\ldots,i_1)
+ r^{0}(i),
\end{equation}
for all $(s,t)\in\cs_{2,T}$. In the previous formula, we assume
$\zeta^{k}\in\cac_1^\ga(\R^{n}\times\R^{d^k})$, and $r^0$ is a regular part
such that $r\in\cac_2^{N\ga}(\R^n)$. We also suppose that for any $1\le k\le N-2$, the increment $\zeta^k$ can be further decomposed into
\begin{equation}\label{eq:weak-dcp-N-further}
\der \zeta^{k}(i,i_1,\ldots,i_k)=
\sum_{l=1}^{N-1-k} \zeta^{k+l}(i,i_1,\ldots,i_{k+l}) \, \bx^{\bl}(i_{k+l},\ldots,i_{k+1})
+ r^{k}(i,i_1,\ldots,i_k),
\end{equation}
where the remainder $r$ belongs to $\cac_2^{(N-k)\ga}(\R^{n}\times\R^{d^k})$.

\smallskip

As in Section \ref{sec:strato-order-2}, the space of weakly controlled
paths will be denoted by $\cq_{\ga,a}(\R^n)$, and a process
$z\in\cq_{\ga,a}(\R^n)$ can be considered in fact as a tuple
$(z,\zeta^{1},\ldots,\zeta^{N-1})$. The natural semi-norm on $\cq_{\ka,a}(\R^n)$ is given
by
\begin{eqnarray*}
\cn[z;\cq_{\ga,a}(\R^n)]
&=&
\cn[z;\cac_1^{\ga}(\R^n)]
+ \sum_{k=1}^{N-1}
\cn[\zeta^{k};\cac_1^{\infty}(\R^{n}\times\R^{d^k})]\\
&&\quad + \cn[\zeta^{k};\cac_1^{\ga}(\R^{n}\times\R^{d^k})]
+\sum_{k=0}^{N-1}\cn[r^k;\cac_2^{(N-k)\ga}(\R^n)],
\end{eqnarray*}
with $\cn[g;\cac_1^{\ka}]$ defined by
(\ref{def:hnorm-c1}) and
$\cn[\zeta;\cac_1^{\infty}(V)]=\sup_{0\le s\le T}|\zeta_s|_V$.
\end{definition}

\smallskip

The decomposition of $f(x)$ as a controlled process for a smooth enough function $f$ can now be read as follows:
\begin{proposition}\label{prop:cp-weak-phi-N}
Let $x$ be a path satisfying Hypothesis \ref{hyp:rough-intro} and let  $f:\R^d\to\R$ be a $C_b^N$ function such that $f(x_0)=a$, and set $z=f(x)$. Then $z\in\cq_{\ga,a}$,  and it can be decomposed into
\begin{equation}\label{eq:dcp-f-z-N}
\der z=\sum_{k=1}^{N-1} \zeta^{k}(i_1,\ldots,i_k) \, \bx^{\bk}(i_k,\ldots,i_1)
+ r^{0},
\end{equation}
with
\begin{equation}\label{desc}
\zeta^{k}_{s}(i_1,\ldots,i_k)= \partial_{i_k\cdots i_1}^{k} f(x_s),
\qquad
r_{st}^{0}= [\der f(x)]_{st}-
\sum_{k=1}^{N-1}\frac{\partial_{i_1\cdots i_k}^{k} f(x_s)}{k!} \,
\prod_{j=1}^{k} \bx_{st}^{\1}(i_j),
\end{equation}
and
\begin{equation}\label{eq:def-r-k}
r^{k}_{st}(i_1,\ldots,i_k)=
\der \lc\partial_{i_1\cdots i_k}^{k} f(x)\rc_{st}-
\sum_{p=1}^{N-k-1}
\frac{\partial_{i_1\cdots i_k j_1\cdots j_p}^{k+p} f(x_s)}{p!} \,
\prod_{q=1}^{p} \bx_{st}^{\1}(j_q),
\end{equation}
where we recall our convention (\ref{eq:convention-partial-deriv}) for $\partial^{k}_{i_1\cdots i_k}f$.
Furthermore,
\begin{equation}\label{eq:bnd-phi-z-N}
\cn[z;\cq_{\ga,a}]\le c_{f,T}\lp 1+\cn^N[x;\cac_1^{\ga}(\R^d)] \rp.
\end{equation}
\end{proposition}

\begin{proof}
The algebraic part of the assertion is obtained by combining a simple Taylor expansion and our geometric assumption (\ref{eq:geom-rough-path}). Indeed, Taylor's expansion directly yields
\begin{equation}\label{eq:taylor-f-x-N}
[\der f(x)]_{st}= D_{st}
+r_{st}^{0},
\quad\mbox{with}\quad
D_{st}=\sum_{k=1}^{N-1}\frac{\partial_{i_1\cdots i_k}^{k} f(x_s)}{k!} \,
\prod_{j=1}^{k} \bx_{st}^{\1}(i_j),
\end{equation}
where $r^{0}\in\cac_2^{N\ga}$. Moreover,
appealing to (\ref{eq:geom-rough-path}), we have
\begin{eqnarray}
D_{st}&=&
\sum_{k=1}^{N-1}\frac{1}{k!}\sum_{i_1,\ldots,i_k=1}^{d} \partial_{i_1\cdots i_k}^{k} f(x_s) \,
\prod_{j=1}^{k} \bx_{st}^{\1}(i_j)\\
&=&\sum_{k=1}^{N-1}\frac{1}{k!}\sum_{i_1,\ldots,i_k=1}^{d}
\partial_{i_1\cdots i_k}^{k} f(x_s) \,
\sum_{\si\in\Sigma_k} \bx^{\bk}(i_{\si(1)},\ldots,i_{\si(k)}),
\end{eqnarray}
and invoking the symmetry properties for the derivatives of $f$ we obtain
\begin{eqnarray}\label{eq:prod-x1-xk}
D_{st}&=&
\sum_{k=1}^{N-1}\frac{1}{k!}\sum_{i_1,\ldots,i_k=1}^{d}
\sum_{\si\in\Sigma_k} \partial_{i_{\si(k)}\cdots i_{\si(1)}}^{k} f(x_s) \,
\bx^{\bk}(i_{\si(1)},\ldots,i_{\si(k)}) \notag\\
&=&\sum_{k=1}^{N-1}\sum_{i_1,\ldots,i_k=1}^{d}
\partial_{i_{k}\cdots i_{1}}^{k} f(x_s) \,
\bx^{\bk}(i_{1},\ldots,i_{k}).
\end{eqnarray}
Going back to our convention on repeated indices, we end up with
\begin{equation*}
[\der f(x)]_{st}
=\sum_{k=1}^{N-1} \partial_{i_{k}\cdots i_{1}}^{k} f(x_s) \,
\bx^{\bk}(i_{1},\ldots,i_{k})
+r^{0}_{st},
\end{equation*}
which is the announced formula (\ref{eq:dcp-f-z-N}). Further expansions of the coefficients $\partial_{i_{k}\cdots i_{1}}^{k} f(x)$, leading to a relation of type (\ref{eq:weak-dcp-N-further}), are performed in the same way and are left to the reader for sake of conciseness.

\vspace{0.3cm}

In order to give an estimate for $\cn[z;\cq_{\ga,a}(\R^n)]$, one has
of course to establish bounds for $\cn[z;\cac_1^{\ga}(\R^n)]$,
$\cn[\zeta^{k};$ $\cac_1^{\ga}(\R^d)]$,
$\cn[\zeta^{k};\cac_1^{\infty}(\R^d)]$ and $\cn[r^{k};\cac_2^{2\ga}]$. These
estimates are readily obtained from the expressions in
decomposition (\ref{desc}), and details are left to the reader. The
analytic bound (\ref{eq:bnd-phi-z-N}) is also obtained in a
straightforward manner.

\end{proof}
\begin{remark}\label{finitud2}
As in the case $n=2$, we point out that the algebraic conclusion of
this proposition is still true for a $f\in\cac^N(\R^d)$, without boundedness restrictions. However, inequality (\ref{eq:bnd-phi-z-N}) would
%be harder to obtain in this situation.
take a different form, since the multiplicative constants depend on
the derivatives of $f$ composed with $x$.
\end{remark}

\subsection{Integration of weakly controlled paths}
The formula which defines the integral of a controlled process with respect to $x$ is now defined similarly to the one
in Theorem~\ref{intg:mdx}, in spite of the roughness of $x$.
\begin{theorem}\label{prop:intg-mdx-N}
For a given $\ga>0$ with $\lfloor 1/\ga \rfloor=N$ (that is,
$1/(N+1)<\ga\le 1/N$), let $x$ be a process satisfying Hypothesis \ref{hyp:rough-intro}. Furthermore,  let $m\in\cq_{\ga,b}(\R^{d})$ with
decomposition $m_0=b\in\R^d$ and
\begin{equation}\label{eq:dcp-m-N}
\der m(i)=\sum_{k=1}^{N-1} \mu^{k}(i,i_1,\ldots,i_k) \, \bx^{\bk}(i_k,\ldots,i_1)
+ r^{0}(i),
\end{equation}
where the increments $\mu^{k}$ satisfy the further assumptions of Definition \ref{def:weak-ps-N}.
Define $z$ by $z_0=a\in\R$ and
\begin{align}\label{eq:dcp-mdx-N}
&\der z=
m(i) \, \bx^{\1}(i) + \sum_{k=1}^{N-1} \mu^{k}(i,i_1,\ldots,i_k) \bx^{\bk+\1}(i_k,\ldots,i_1,i) \\
&-\laa\lp \sum_{k=0}^{N-2} r^{k}(i,i_1,\ldots,i_k)
\bx^{\bk+\1}(i_k,\ldots,i_1,i) +\der\mu^{N-1}(i,i_1,\ldots,i_{N-1})
\bx^{\bnn}(i_{N-1},\ldots,i_1,i) \rp . \notag
\end{align}
Finally, set
$$
\cj_{st}(m\, dx)=\ist \left\langle m_u, \, dx_u\right\rangle_{\R^d} \triangleq (\der z)_{st}.
$$
Then:

\smallskip

\noindent
\emph{(1)}
$z$ is well-defined as an element of $\cq_{\ka,a}(\R)$, and coincides
with the Riemann integral of $z$ with respect to $x$ whenever these two functions are smooth.

\smallskip

\noindent
\emph{(2)}
The semi-norm of $z$ in $\cq_{\ka,a}(\R)$ can be estimated as
\begin{equation}\label{eq:bnd-norm-imdx-N}
\cn[z;\cq_{\ga,a}(\R)]\le c_{x} \lp 1 +
\cn[m;\cq_{\ga,b}(\R^{d})]\rp,
\end{equation}
for a positive constant $c_{x}$ which can be bounded as $c_x\le c \sum_{k=1}^{N} \cn[\bx^{\bk};\, \cac_{2}^{k\ga}]$, where $c$ stands for a universal constant.

\smallskip

\noindent
\emph{(3)}
It holds
\begin{equation}\label{eq:rsums-imdx-N}
\cj_{st}(m\, dx) =\lim_{|\Pi_{st}|\to 0}\sum_{q=0}^{n-1}\Big[
m_{t_{q}}(i) \, \bx^{\1}_{t_{q}, t_{q+1}}(i) + \sum_{k=1}^{N-1}
\mu^{k}_{t_{q}}(i,i_1,\ldots,i_k)\, \bx^{\bk+\1}_{t_{q}
t_{q+1}}(i_k,\ldots,i_1,i)\Big]
\end{equation}
for any $0\le s<t\le T$,
where the limit is taken over all partitions
$\Pi_{st} = \{s=t_0,\dots,t_n=t\}$
of $[s,t]$, as the mesh of the partition goes to zero.
\end{theorem}

\begin{proof}
Relying on what has been done at Section \ref{sec:strato-order-2}, we mainly derive here the expression~(\ref{eq:dcp-mdx-N}) for $\cj(m\, dx)$. Once this expression is obtained, the other estimates follow like
in Theorem \ref{intg:mdx}, except for the higher number of terms which have to be taken care of.

\smallskip

Hence let us assume for the moment that $m$ and $x$ are smooth functions, and try to define $\cj(m\, dx)$ in an appropriate way for generalizations to rougher cases: one can write, using decomposition (\ref{eq:dcp-m-N}),
\begin{eqnarray}\label{eq-expr-I-m-dx-N}
\cj(m\, dx)&=&
m(i) \, \bx^{\1}(i) + \cj(\der m(i) \, dx) \\
&=&m(i) \, \bx^{\1}(i)
+\sum_{k=1}^{N-1} \mu^{k}(i,i_1,\ldots,i_k) \bx^{\bk+\1}(i_k,\ldots,i_1,i)
+\cj(r^{0} \, dx).
\end{eqnarray}
Hence, like for equation (\ref{eq:di-rdx}), one can deduce that
\begin{multline*}
\der(\cj(r^{0} \, dx))
=-\der m(i) \, \bx^{\1}(i)
- \sum_{k=1}^{N-1} \der \mu^{k}(i,i_1,\ldots,i_k) \bx^{\bk+\1}(i_k,\ldots,i_1,i) \\
+\sum_{k=1}^{N-1} \mu^{k}(i,i_1,\ldots,i_k) \der\bx^{\bk+\1}(i_k,\ldots,i_1,i).
\end{multline*}
We now plug relation (\ref{eq:dcp-m-N}) for $\der m$, relation
(\ref{eq:weak-dcp-N-further}) for $\der\mu^k$ and the multiplicative
relation (\ref{eq:multiplicativity}) for $\der\bx^{\bk+\1}$ into the
latter equation. This yields
\begin{align}\label{eq:exp-I-r0-dx-N}
&\der(\cj(r^{0} \, dx))=
- \sum_{k=1}^{N-1} \mu^{k}(i,i_1,\ldots,i_k) \, \bx^{\bk}(i_k,\ldots,i_1) \, \bx^{\1}(i)
-r^{0}(i) \, \bx^{\1}(i)  -M  \notag\\
&-\der\mu^{\bnn-\1}(i,i_1,\ldots,i_{N-1}) \, \bx^{\bnn}(i_{N-1},\ldots,i_1,i)
- \sum_{k=1}^{N-2} r^{k}(i,i_1,\ldots,i_k) \, \bx^{\bk+\1}(i_k,\ldots,i_1,i) \notag\\
&+\sum_{k=1}^{N-1} \mu^{k}(i,i_1,\ldots,i_k) \sum_{l=1}^{k}
\bx^{\bl}(i_k,\ldots,i_{k-l+1}) \,
\bx^{\bk+\1-\bl}(i_{k-l},\ldots,i_1,i),
\end{align}
where
\begin{equation*}
M=\sum_{k=1}^{N-2} \Big(\sum_{l=1}^{N-1-k}
\mu^{k+l}(i,i_1,\ldots,i_{k+l}) \,
\bx^{\bl}(i_{k+l},\ldots,i_{k+1})\Big) \,
\bx^{\bk+\1}(i_k,\ldots,i_1,i).
\end{equation*}
Moreover, a simple change of index allows to write
$$
M=\sum_{q=2}^{N-1} \mu^{q}(i,i_1,\ldots,i_q) \, \sum_{l=1}^{q-1}
\bx^{\bl}(i_{q},\ldots,i_{q-l+1}) \,
\bx^{\bq+\1-\bl}(i_{q-l},\ldots,i_1,i),
$$
and hence (\ref{eq:exp-I-r0-dx-N}) simplifies into
\begin{multline}\label{eq:exp-der-I-r0-dx-N}
\der(\cj(r^{0} \, dx))  \\= -\sum_{k=0}^{N-2}
r^{k}(i,i_1,\ldots,i_k) \bx^{\bk+\1}(i_k,\ldots,i_1,i)
-\der\mu^{N-1}(i,i_1,\ldots,i_{N-1})
\bx^{\bnn}(i_{N-1},\ldots,i_1,i).
\end{multline}
It is now readily checked that the operator $\laa$ can be applied to the latter increment whenever $x\in\cac_1^{\ga}$ and generates a weakly geometric rough path. Putting together relations~(\ref{eq-expr-I-m-dx-N}) and (\ref{eq:exp-der-I-r0-dx-N}) we thus end up with expression (\ref{eq:dcp-mdx-N}) for the integral $\cj(m\, dx)$.

\smallskip

The analytic bounds are now a matter of standard calculations, and are left to the reader for sake of conciseness.

\end{proof}

\subsection{It\^{o}-Stratonovich formula}
\label{sec:stoch-roughsub}

Now that we know how to define integrals of controlled processes with respect to $x$, our change of variable formula for $f(x)$ is obtained quite in the same way as in the second order setting. The formula can then be read as follows:
\begin{theorem}\label{ito-strato-N}
For a given $\ga>0$ with $\lfloor 1/\ga \rfloor=N$, let $x$ be a
process satisfying the regularity, multiplicative and geometric
hypotheses of Section \ref{sec:intro}. Let $f$ be  a
$C^{N+1}(\R^{d};\R)$ function. Then
\begin{equation}\label{itoform1}
\lc \der(f(x)) \rc_{st}= \cj_{st}\lp \nabla f(x) \, dx \rp
=\ist \lla  \nabla f(x_u), \, dx_u \rra_{\R^d},
\end{equation}
where the integral above has to be understood in the sense of
Theorem \ref{prop:intg-mdx-N}.
Moreover, %it holds
%\begin{equation}\label{eq:rsums-strato}
%\cj_{st}\lp \nabla f(x) \, dx \rp =\lim_{|\Pi_{st}|\to 0} S_{\Pi}(f)
%\end{equation}
%where
\begin{align}\label{eq:def-rsums-strato}
&\cj_{st}\lp \nabla f(x) \, dx \rp=\lim_{|\Pi_{st}|\to
0}\sum_{q=0}^{n-1}\Big[
\partial_i f(x_{t_{q}}) \, \bx^{\1}_{t_{q}, t_{q+1}}(i)
+ \sum_{k=1}^{N-1} \partial_{i_k\ldots i_1
i}^{k+1}f(x_{t_q}) \, \bx^{\bk+\1}_{t_{q}t_{q+1}}(i_k,\ldots
,i_1,i)
\Big] \notag\\
&=\lim_{|\Pi_{st}|\to 0}\sum_{q=0}^{n-1}\Big[
\partial_i f(x_{t_{q}}) \, \bx^{\1}_{t_{q} t_{q+1}}(i)
+ \sum_{k=1}^{N-1} \frac{1}{k!}\partial_{i_k\ldots i_1
i}^{k+1}f(x_{t_q}) \,
\bx^{\1}_{t_{q}t_{q+1}}(i_k)\,\cdots\,\bx^{\1}_{t_{q}t_{q+1}}(i_1)\,\bx^{\1}_{t_{q}t_{q+1}}(i)
\Big]
\end{align}
for any $0\le s<t\le T$,
where the limit is taken over all partitions
$\Pi_{st} = \{s=t_0,\dots,t_n=t\}$
of $[s,t]$, as the mesh of the partition goes to zero.
\end{theorem}

\begin{proof}
The proof goes exactly along the same lines as for Proposition
\ref{prop:ito-strat-form-2}. The first expression for $\cj(\nabla
f(x)dx)_{st}$ given in (\ref{eq:def-rsums-strato}) is a consequence
of the decomposition as a weakly controlled path of $\nabla f(x)$
and Theorem \ref{prop:intg-mdx-N}. The second one follows from the
first one by using  Schwarz rule and the geometric property
(\ref{eq:geom-rough-path}).

\smallskip

On the other hand, applying Taylor's formula up to order $N$ to the
decomposition
$$f(x_t)-f(x_s)=\sumq f(x_{t_{q+1}})-f(x_{t_q}),$$
comparing it with (\ref{eq:def-rsums-strato}) and using that
$(N+1)\ga>1$, one obtain easily the Stratonovich type formula.

%:
% we compare the decompositions of $\der(f(x))$ and $\cj( \nabla f(x) \, dx )$ and show that
% they coincide up to some terms with H\"{o}lder exponent $>1$.
%
%\smallskip
%
%More specifically, the Taylor expansion of $f(x)$ up to order $N$ gives, exactly as for relation (\ref{eq:dcp-f-z-N}),
%\begin{equation}\label{eq:dcp-f-x-N+1}
%\der(f(x))=\sum_{k=1}^{N} \zeta^{k}(i_1,\ldots,i_k) \, \bx^{\bk}(i_k,\ldots,i_1)
%+ r^{\sharp},
%\end{equation}
%where $\zeta^{k}_{s}(i_1,\ldots,i_k)= \partial_{i_k\cdots i_1}^{k}
%f(x)$ and $r^{\sharp}\in\cac_2^{(N+1)\ga}$. The decomposition of
%$\cj( \nabla f(x) \, dx )$ is obtained by a direct application of
%Theorem \ref{prop:intg-mdx-N}.
% One immediately sees that it coincides with the one given by (\ref{eq:dcp-f-x-N+1}),
% up to the terms inside the operator $\laa$, which lye in $\cac_2^{(N+1)\ga}$. Straightforward details are
%  left to the reader.
%

\end{proof}

\section{Skorohod type formula via Malliavin calculus}
\label{sec:sko-mall}

We take now a completely different direction in our considerations: the pointwise point of view which had been adopted previously is abandoned in this section, and we try to construct an integral with respect to a (Gaussian) process $x$ by means of stochastic analysis tools. We then prove that for any $0\le s<t<\infty$, the function $\1_{[s,t)}\nabla f(x)$ is in the domain of an extended divergence operator with respect to $x$, and prove an associated Skorohod type formula. As we shall see, this mainly stems from an extension of \cite{CH} to the $d$-dimensional case, which is allowed thanks to the symmetries of $\nabla f(x)$.

\subsection{Preliminaries on Gaussian processes}
\label{Mall}

>From now on, we specialize our setting to a centered Gaussian
process $x=(x(1),\ldots,x(d))$ with i.i.d coordinates, and
covariance function
\begin{equation}\label{eq:def-cov-x}
R_{st}:=\be[x_s(1)x_t(1)],
\quad\mbox{and}\quad
R_t:=\be[|x_t(1)|^2]=R_{tt},
\qquad s,t\in [0,T].
\end{equation}
We will add later some hypotheses on these functions.
% satisfies:
%\begin{equation*}
%\int_0^T |\partial_v R_v| \, dv <\infty.
%\end{equation*}
We can also assume that $x_0(j)=0$.
\def\cS{{\mathcal S}}

\smallskip

The Gaussian integration theory is based on a completion (in $L^2(\oom)$) of elementary integrals with respect to $x$, which can be summarized as follows (see \cite{Nu-bk} for more details): consider the space of $d$-dimensional elementary functions
\begin{eqnarray*}
\cS=\Big\{ f=(f_1,\ldots,f_d);\,\,f_j=\sum_{i=0}^{n_j-1} a_i^j
\1_{[t_i^j, t_{i+1}^j)}\,, \quad
0=t_0<t_1^j<\cdots<t_{n_j-1}^j<t_{n_j}^j=T,\\
\text{ for }j=1,\ldots,d\Big\}\,.
\end{eqnarray*}
%For two indicate functions in $\cS$  we define
%\[
%\langle \1_{[0, s)}\,, \1_{[0, t)}\rangle_{\mathcal H}=R_{st}\,.
%\]
%This relation can be extended  to all elements in $\cS$ by
%linearity. Thus we introduced a scalar product in $\cS$.  Let
%$\mathcal H$ be Hilbert space from the completion of $\cS$ with
%respect to the above  scalar product.
For any element $f$ in $\cS$, we define the integral of first order
of $f$ with respect to $x$ as
\[
I_1(f):=\sum_{j=1}^d\sum_{i=0}^{n_j-1} a_i^j (x_{t_{i+1}^j}(j)
-x_{t_i^j}(j))\,.
\]

\smallskip

For $\theta:\R\rightarrow \R$, and $j\in\{1,\ldots,d\}$, denote by
$\theta^{[j]}$ the function with values in $\R^d$ having all the
coordinates equal to zero, except the $j\textsuperscript{th}$ coordinate which is equal
to $\theta$. It is readily seen that
$$\be[I_1(\1_{[0,s)}^{[j]})I_1(\1_{[0,t)}^{[k]})]=\1_{(j=k)}R_{st}.$$
 So, we can
define for some indicator functions of $\cs$ the following symmetric
and semi-definite  form
$$\langle \1_{[0,s)}^{[j]},\1_{[0,t)}^{[k]}\rangle_{\cs}=\1_{(j=k)}R_{st},$$
and extend it to all elements of $\cs$ by linearity. If we identify
two functions $f$ and $g$ in $\cs$ when $\langle
f-g,f-g\rangle_{\cs}=0$, then $\langle\cdot,\cdot\rangle_{\cs}$
becomes an inner product on $\cs$ (actually, on the quotient space
obtained by this identification). Therefore, for $f$ and $g$ in
$\cs$ we have that
$$\be[I_1(f)I_1(g)]=\langle f,g\rangle_{\cs}$$
and $I_1$ defines an isometric map from $\cs$, endowed with the
inner product $\langle\cdot,\cdot\rangle_{\cs}$ into a subspace of
$L^2(\Omega)$. This map can be extended in the standard way to an
isometric map, denoted also as $I_1$, from a real Hilbert space that
we will denote by $\ch$ into a closed subspace of $L^2(\Omega)$.
>From now on, denote the inner product of this extended isometry by
$\langle \cdot,\cdot\rangle_{\ch}.$ We will assume that $\ch$ is a
separable Hilbert space (which is satisfied whenever $R_{st}$ is
continuous).

\smallskip

Let $\{ e_1,   e_2, \cdots \}$  be an orthonormal basis  of
$\mathcal H$ and let  $\hatotimes$ denote  the symmetric tensor
product.  Then
\begin{equation}\label{e.2.2}
f_n=\sum_{\rm { finite}} f_{i_1, \cdots, i_n}
e_{i_1}\hatotimes\cdots \hatotimes e_{i_n}, \quad  f_{i_1, \cdots,
i_n}\in \R
\end{equation}
is an element of $\mathcal H^{\hatotimes n}$ with the Hilbert norm
\begin{equation}\label{eq:def-norm-Hn}
\|f_n\|_{\mathcal H^{\hatotimes n}}^2 =\sum_{\rm { finite}}
|f_{i_1, \cdots, i_n}|^2\, .
\end{equation}
Moreover, $\mathcal H^{\hatotimes n}$  is the completion  of  all the elements like (\ref{e.2.2}) with respect to the norm~(\ref{eq:def-norm-Hn}).

\smallskip

For an element $f_n\in \mathcal H^{\hatotimes n}$, the multiple
It\^o integral of order $n$ is well-defined. First, any  element of
the form given by (\ref{e.2.2}) can be rewritten as
\begin{equation}\label{htensor}
f_n=\sum_{\rm { finite}} f_{j_1 \cdots j_m} e_{j_1}^{\hatotimes
k_1}\hatotimes\cdots \hatotimes e_{j_m}^{\hatotimes
k_m},\end{equation} where the $j_1,\ldots, j_m$ are different and
$k_1+\cdots+k_m=n$. Then, if $f_n\in \ch^{\hatotimes n}$  is given under the form (\ref{htensor}), define its multiple integral as:
\begin{equation}\label{eq:def-mult-intg}
I_n(f_n)=\sum_{\rm { finite}} f_{j_1, \cdots, j_m} H_{k_1}( I_1(e_{j_1}))\cdots
H_{k_m}(I_1( e_{j_m})),
\end{equation}
where $H_k$ denotes the $k$-th normalized
Hermite polynomial given by
$$H_k(x)=(-1)^k
e^{\frac{x^2}2}\frac{d^k}{dx^k}e^{-\frac{x^2}2}=\sum_{j\le k/2}
\frac{(-1)^j k!}{2^j\,j!\,(k-2j)!} x^{k-2j}.$$ It holds that the
multiple integrals of different order are
orthogonal and that
$$\be|I_n(f_n)|^2=n!\,\|f_n\|_{\ch^{\hatotimes n}}^2.$$
This last isometric property allows to extend the multiple integral
for a  general
$f_n\in\ch^{\hatotimes n}$
by $L^2(\Omega)$ convergence (notice once again that this kind of closure is different in spirit from the pathwise convergences considered at Sections \ref{sec:strato-order-2} and \ref{sec:strato-order-N}). Finally, one can define the
integral of $f_n\in\ch^{\otimes n}$ by putting
$I_n(f_n):=I_n(\tilde f_n),$
where $\tilde f_n\in\ch^{\hatotimes n}$ denotes the symmetrized version of
$f_n$. Moreover, the
chaos expansion theorem states that any square integrable random
variable $F\in L^2(\Omega,\mathcal G,P)$, where $\mathcal G$ is the
$\sigma$- field generated by $x$,
 can be written as
\begin{equation}\label{chaos}
F=\sum_{n=0}^\infty I_n(f_n)
\quad\mbox{with}\quad
\EE[F^2]=\sum_{n=0}^\infty  n! \|f_n||_{\mathcal H^{\hatotimes
n}}^2\,.
\end{equation}

\smallskip

We will introduce now the (iterated) derivative and  divergence
operators of the Malliavin calculus. We denote by $\mathcal
C^{\infty}_p(\R^n)$ the set of infinitely continuously
differentiable functions $f:\R^n\rightarrow \R$ such that $f$ and
all its partial derivatives have polynomial growth. Let $\textbf S$
denote the class of smooth random variables of the form
\begin{equation}\label{smooth}F=f(I_1(h_1),\ldots,  I_1(h_n)),
\end{equation}
where
$f\in\mathcal C^{\infty}_p(\R^n)$, $h_1,\ldots, h_n$ are in $\ch$,
and $n\ge 1$. The derivative of a smooth random variable $F\in\textbf S$ of the form
(\ref{smooth}) is the $\ch$-valued random variable given by
\begin{equation}\label{deriv}
DF=\sum_{i=1}^n \partial_i f(I_1(h_1),\ldots,  I_1(h_n))h_i,
\end{equation}
where $\partial_i$ denotes as usual $\frac{\partial}{\partial x_i}$. One can also define
for $h\in\ch$ and $F\in\mathbf S$ the derivative of $F$ in the
direction of $h$ as
$D_hF=\langle F,h\rangle_{\ch}.$

\smallskip

The iteration of the operator $D$  is defined in such a way that
for a smooth random variable $F\in\textbf S$ the iterated derivative $D^k F$ is
a random variable with values in $\ch^{\otimes k}$. We also consider
for $h^k\in\ch^{\otimes k}$ the $k$-th derivative of $F$ in the
direction of $h^k$ defined as
$$D^k_{h^k}F=\langle D^kF,h^k\rangle_{\ch^{\otimes k}}.$$
Let us fix now a notation for the domain of the iterated derivative $D^k$:
for every $p\ge 1$ and any natural number $k\ge 1$ we introduce the
seminorm on $\textbf S$ given by
$$\|F\|_{k,p}=\left[\be\big(|F|^p\big)+\sum_{j=1}^k \be\big(\|D^j
F\|_{\ch^{\otimes j}}^p\big)\right]^{\frac1{p}}.$$ It is well-known that the operator
$D^k$ is closable from $\textbf S$ into $L^p(\Omega;\ch^{\otimes
k})$.  We will denote by $\mathbb D^{k,p}\,$ the completion of the
family of smooth random variables $\textbf S$ with respect to the
norm $\|\cdot\|_{k,p}\,$. We will also refer the space $\mathbb
D^{k,2}$ as the domain of the operator $D^k$ and denote it by
$\dom\, D^k$. If $F$ has the chaotic representation (\ref{chaos}),
we have that
$$\be\Big(\|D^kF\|_{\ch^{\otimes
k}}^2\Big)=\sum_{n=k}^{\infty}n(n-1)\cdots(n-k+1)\,n!\|f_n\|_{\ch^{\hatotimes
n}}^2$$
and a useful characterization of  $\dom\, D^k$ is the following: $F\in\dom\, D^k$ if and only if
$$\sum_{n=1}^{\infty} n^k\,n!\,\|f_n\|_{\ch^{\hatotimes
n}}^2<\infty.$$

We will denote by $\delta^{\di}$ the adjoint of the operator $D$
(this operator is also referred as the {\it divergence operator}) and more generally,  we denote by $\delta^{\di k}$ the adjoint of $D^k$. The operator $\delta^{\di k}$ is closed
and its domain, denoted by $\dom \,\delta^{\di k}$, is the set of
$\ch^{\otimes k}$-valued square integrable random variables $u\in
L^2(\Omega;\ch^{\otimes k})$ such that
$$|\be(\langle D^k F,u\rangle_{\ch^{\otimes k}}|\le C\,\|F\|_2,$$
for all $F\in\dom\, D^k$, where $C$ is some constant depending on
$u$. Moreover, for $u\in\dom\,\delta^{\di k}$,  $\delta^{\di k}(u)$
is the element of $L^2(\Omega)$ characterized by the duality
relationship:
\begin{equation}\label{duality}
\be(F\delta^{\di k}(u))=\be(\langle D^k F,u\rangle_{\ch^{\otimes
k}}),
\end{equation}
for any $F\in\dom\, D^k$. For $u\in\dom \,\delta^{\di}$, the random variable $\delta^{\di}(u)$ is usually called \emph{Skorohod integral} of $u$, because it coincides with the usual integral of $u$ with respect to $x$ for a large class of elementary processes $u$ (see \cite{Nu-bk} for further details).

\subsection{An operator associated to $\mathbf{x}$}

Along this section we will consider a $d$-dimensional continuous  process satisfying the following set of assumptions:
\begin{hypothesis}\label{hyp:cov-x}
The process $x=(x(1),\ldots,x(d))$ is a centered Gaussian process with i.i.d. coordinates. Letting $R_{st}$ and $R_{t}$ being defined as in (\ref{eq:def-cov-x}), we suppose that those two functions  are continuous and the following two conditions hold:
\begin{enumerate}
\item The variance function
$R_t:=R_{tt}$  is differentiable at any point $t\in (0,T)$ and
satisfies that
$$\int_0^T |R'_t|\,dt<\infty.$$
\item The first partial derivative $\partial_s R_{st}$ of $R_{st}$ is
well-defined a.e. on $[0,T]^2$ and
verifies
\begin{equation}\label{eq-div-2}
\iott\iott\lln \partial_s R_{sy}\rrn ds dy<\infty.
\end{equation}
\end{enumerate}
\end{hypothesis}
\noindent
We will try now to identify a useful operator for our future Gaussian computations.

\smallskip

Let  $\vp=(\vp(1),\ldots,\vp(d))\in\lp\mathcal D_T\rp^d$, where
$\mathcal D_T$ is the space of $\mathcal C^{\infty}$ functions with
compact support contained in $(0,T)$. We have that (see for instance
\cite{jo}, where the 1-dimensional case is considered) that
$\vp\in\ch$ and that
$$I_1(\vp)=-\iott\langle x_s,\,\vp'_s\rangle\,ds,$$
where $\langle\cdot,\cdot\rangle$ denotes the ordinary Euclidean
product in $\R^d$. Moreover,  $\lp\cald\rp^d$ is a dense subset of
$\ch$. From now on, we use also the notation $x(f)$ for $I_1(f)$.

\smallskip

Given a function $h:[0,T]\to\R$, recall that $h\crj$ denotes the
function with values in $\R^d$ in which all the coordinates except
the $j$-th one are equal to $0$ and the $j$-th coordinate equals to
$h$. Therefore, for $\beta\in\cald$ and $0\le a<b\le T$, we have
that
\begin{eqnarray}\label{eq-div-1}
\lla  \1_{[a,b)}\crl\,,\,\beta\crj\rra_{\ch}&=&\be\lc
I_1(\1_{[a,b)}\crl)I_1(\beta\crj)\rc=-\be\lc\lp x_b(l)-x_a(l)\rp\iott\langle x_t\,,\lc\beta\crj \rc'_t\rangle\,dt\rc\nonumber\\
&=&-\,\1_{(j=l)}\,\iott\lp R_{bt}-R_{at}\rp
\beta_t'\,dt=-\,\1_{(j=l)}\,\iott\lp\int_a^b\partial_{s}R_{st}\,ds\rp\beta'_t\,dt\nonumber\\
&=&-\,\1_{(j=l)}\,\int_a^b\lp\iott
\partial_{s}R_{sy}\,\beta'_y\,dy\rp ds.
\end{eqnarray}
We will consider the first iterated integral appearing  on the right
hand side of (\ref{eq-div-1}) as a linear operator defined on
$\mathcal D_T$. That is, we consider for $s\in [0,T]$ and $\beta
\in\mathcal D_T $, the following function:
$$\mathbf A\beta(s)=:-\iott \partial_{s}R_{sy}\,\beta_y'\,dy.
$$
We will suppose from now on that the following hypothesis holds.
\begin{hypothesis}\label{l2-integr}
For any $\beta\in\mathcal D_T$, $\mathbf A\beta\in L^2([0,T])$.
\end{hypothesis}
\begin{remark}
Condition (\ref{eq-div-2}) on $R_{st}$, stated in Hypothesis \ref{hyp:cov-x},
implies that $\mathbf A\beta$ belongs to $L^1([0,T])$ whenever $\beta\in\mathcal{D}_T$. We have
imposed the additional condition  $A\beta\in L^2([0,T])$ in order to guarantee the
integrability of many terms appearing in the sequel.
Although one can  weaken Hypothesis~\ref{l2-integr}, this would complicate some
of the next  statements. We have thus chosen to impose it for the sake of
simplicity.
\end{remark}

\begin{example}\label{ex:fbm}
Hypothesis \ref{l2-integr} is satisfied by the fractional Brownian
motion. In fact,
\begin{eqnarray*}
\mathbf A\beta(s)&=&-\iott \frac{\partial}{\partial s} R_{sy}\,
\beta'(y)\,dy=-\iott H\lp s^{2H-1}-|s-y|^{2H-1}\text{sign}(s-y)\rp
\, \beta'(y)\,dy\\&=& \iott H\lp |s-y|^{2H-1}\text{sign}(s-y)\rp
\beta'(y)\,dy,
\end{eqnarray*}
because
$\beta\in\mathcal D_T$. And from this, it is easily seen that
$\mathbf A \beta\in L^{\infty}([0,T])$ if $\beta\in\mathcal{D}_T$.
\end{example}

\smallskip

Let us now relate our operator $A$ to the inner product in $\ch$:
equation  (\ref{eq-div-1})  tells us that for any elementary function
$g=(g(1),\ldots,g(d))\in\mathcal S$ and $\beta\in\cald$ we have that
$$\lla\beta\crj\,,\,g(l)\crl\rra_\ch=\1_{(j=l)}\iott g_s(l)\mathbf A\beta(s)\,ds.$$
Since for $\vp\in\lp\cald\rp^d,\,$ $g\in\mathcal S$, we have
$\vp=\sum_{j=1}^d \vp(j)\crj$ and $g=\sum_{l=1}^d g(l)\crl$,
we obtain that
\begin{equation}\label{eq-div-3-smooth}
\lla\vp\,,\,g\rra_\ch=\sum_{j=1}^d\iott g_s(j){\mathbf
A}\vp(j)(s)\,ds=\iott \lla g_s\,,\,{\mathbf A}\vp(s)\rra\,ds,
\end{equation}
where we use the notation ${\mathbf A}\vp=({\mathbf
A}\vp(1),\ldots,{\mathbf A}\vp(d))$. Extending this last relation by continuity, the following useful representation for the inner product in $\ch$ is readily obtained:
\begin{lemma}
For any $g\in\ch\cap (L^2([0,T]))^d$ and $\vp\in\cald$, one can write
\begin{equation}\label{eq-div-3}
\lla\vp\,,\,g\rra_\ch=\iott \lla g_s\,,\,{\mathbf A}\vp(s)\rra\,ds,
\end{equation}
where $\langle\cdot,\cdot\rangle$ stands for the inner product in $\R^d$.
\end{lemma}

Going back to our example \ref{ex:fbm}, notice that expression (\ref{eq-div-3})  is similar to the  following one
pointed out in \cite{CH} for the one-dimensional fractional Brownian motion with
Hurst parameter $H<1/2$:
$$
\lla \vp\,,\,g\rra_\ch=c_H^2\iott g(s)\mathbf D_{+}^\alpha \mathbf
D_{-}^\alpha\vp(s)\,ds,$$ where $\alpha=\frac12-H$; $\mathbf
D_{+}^\alpha$  and $\mathbf D_{-}^\alpha$ are the Marchaud
fractional derivatives (see \cite{SKM} for more details about these
objects), and $c_H$ is a certain positive  constant.

\subsection{Extended divergence operator}
Let us take up here the notations of Section \ref{Mall}. Having noticed that fBm
 gives rise to an operator $\mathbf D_{+}^\alpha \mathbf D_{-}^\alpha$ which is a particular case of our operator $\mathbf{A}$ (see Example \ref{ex:fBm2}),
  one can naturally try  to define an extension of the operator
$\ddi$ using similar arguments to those of \cite{CH}. The idea is to
consider first $u\in\rm{Dom}\,\ddi\cap \lp
L^2(\oom\times[0,T])\rp^d$ and $F=\hnj$ where $H_n$ is the n-th
normalized Hermite polynomial, and $\vp\in\lp\cald\rp^d$. Since $\lp
L^2(\oom\times[0,T])\rp^d\equiv L^2\lp\oom;L^2([0,T];\R^d)\rp$ and
$\dom\, \ddi\subset L^2(\oom;\ch)$, we have that $u\in \lp
L^2(\oom\times[0,T])\rp^d\cap\ch$ almost surely. Moreover, $D
\hnju=\hnju\vp\in\lp\cald\rp^d$, a.s.. So, using (\ref{eq-div-3}),
the usual duality relationship between $D$ and $\ddi$ can be written
in the following way:
\begin{multline}\label{eq:ibp-extended}
\be\lc\ddi(u)\hnj\rc=\be\lc\lla u\,,\,D\hnj\rra_\ch\rc=\be\lc\hnju\lla u\,,\vp\rra_\ch\rc \\
=\be\lc\hnju\iott\lla u_s \,,\,{\mathbf
A}\vp(s)\rra\,ds\rc=\iott\lla \be\lc \hnju u_s\rc\,,\, {\mathbf
A}\vp(s)\rra\,ds,
\end{multline}
and this motivates the following definition.
\begin{definition}\label{def:ext-divergence}
We say that $u\in \dom^*\ddi$ if $u\in\lp L^2(\oom\times[0,T])\rp^d$
and there exists an element of $L^2(\oom)$, that will be denoted by
$\ddi(u)$, such that for any $\vp\in\lp\cald\rp^d$ and any $n\ge 0$,
the following is satisfied: \beq\label{eq-div-4}
\be\lc\ddi(u)\hnj\rc= \iott\lla \be\lc \hnju u_s\rc\,,\, {\mathbf
A}\vp(s)\rra\,ds.
\eeq
\end{definition}
\begin{remark}\label{rmk:density}
Since the linear span of the set $\{H_n(\vp):\,n\ge 0,\,\vp\in
(\mathcal D_T)^d\}$ is dense in $L^2(\oom)$, the element
$\delta^{\di}(u)$, if it exists, is uniquely defined.
\end{remark}
\begin{remark}\label{rmk:closure-ext-divergence}
One can easily see from our definition of the extended divergence
that it is a closed operator in the following sense: if $\{u^k\}_{k\in\N}\subset \dom^*\ddi$ and satisfies \textit{(1)} $u^k\rightarrow u$ in $\lp L^2(\oom\times[0,T])\rp^d$ and \textit{(2)} $\ddi(u^k)\rightarrow X$ in $L^2(\oom)$, then $u\in\dom^*\ddi$ and $\ddi(u)=X$.
\end{remark}

We show in the following proposition that the extended operator
$\ddi$ defined above is actually an extension of the
divergence operator of the Malliavin calculus.
\begin{proposition}
The domain $\dom^*\ddi$ is an extension of $\dom\,\ddi$ in the following sense:
$$\dom\,\ddi\cap\lp L^2(\oom\times[0,T])\rp^d=\dom^*\ddi\cap
L^2(\oom;\ch).$$
Furthermore, the extended operator $\ddi$ restricted to
$\dom\,\ddi\cap\lp L^2(\oom\times[0,T])\rp^d$ coincides with the
standard divergence operator.
\end{proposition}
\begin{proof}
If $u\in \dom\,\ddi\cap\lp L^2(\oom\times[0,T])\rp^d$ then $u\in \lp L^2([0,T])\rp^d\cap \ch$ almost surely. Thus~(\ref{eq-div-3}) can be applied to $u$ and (\ref{eq:ibp-extended}) holds true for $\ddi(u)$ (the standard divergence operator). This proves that $\dom\,\ddi\cap\lp
L^2(\oom\times[0,T])\rp^d\subset\textrm{Dom}^*\ddi\cap
L^2(\oom;\ch)$ and that $\ddi$ is an extension of the standard
divergence operator on $\dom\,\ddi\cap\lp
L^2(\oom\times[0,T])\rp^d$.

\smallskip

To see the other inclusion, take $u\in \dom^*\ddi\cap
L^2(\oom;\ch)$. By our Definition \ref{def:ext-divergence} of $\dom^*\,\ddi$, $u$ belongs also to
$\lp L^2(\oom\times[0,T])\rp^d$. We will show that $u\in\dom\,\ddi$.
First, we will prove that the element $\ddi(u)$ defined by the
equality (\ref{eq-div-4}) satisfies, for any $\vp\in\lp\cald\rp^d$
and any $n\ge 0$,  that
\beq\label{eq-div-6}
\be\lc\ddi(u)\hnj\rc=\be\lc\lla u\,,\,D\hnj\rra_\ch\rc.
\eeq
Indeed, since $u\in
\lp L^2([0,T])\rp^d\cap \ch$ a.s. by assumption, we can apply again identity
(\ref{eq-div-3}) and so,
$$\lla u\,,\,D\hnj\rra_\ch=\hnju\iott\lla u_s\,,\,{\mathbf A}\vp(s)\rra\,ds.$$
Hence, using Fubini's theorem and (\ref{eq-div-4}) we end up with
\begin{eqnarray*}
\be\lla u\,,\,D\hnj\rra_\ch=\iott \lla\be\lc\hnju\,u_s\rc\,,\,
{\mathbf A}\vp(s)\rra\,ds=\be\lc\ddi(u)\hnj\rc,
\end{eqnarray*}
which is exactly (\ref{eq-div-6}).

\smallskip

By using density arguments (the
linear space generated by the elements of the form $\hnj$, with
$\vp\in\lp\cald\rp^d$, $n\ge 0$, is dense in $\dom \,{\rm D}$) we
obtain that
$$\be\lc \lla u\,,\,DF\rra_\ch\rc=\be\lc\ddi(u)\,F\rc
$$
for any $F\in \dom \,{\rm D}$, and this finishes the proof.

\end{proof}

\begin{example}\label{ex:fBm2}
Go back to our fBm Example \ref{ex:fbm}, and let us  compare  the extended divergence operator introduced
above with the one defined in \cite{CH}. First of
all, we must point out that in \cite{CH}, the (standard)
divergence operator is presented in  a more general setting than
ours: the divergence can belong to any $L^p(\oom)$, for $p>1$. In our paper, we will only consider this divergence over $L^2$ spaces for sake of conciseness.

\smallskip

According to the computations carried out in \cite{Hu} (see identity (5.30)
of that work), for any element $\psi$ in the space of test functions $\cald$ one has
$$c_H^2\mathbf D_{+}^\alpha
\mathbf D_{-}^\alpha\psi(s)=\iott
H\,|s-y|^{2H-1}\text{sign}(s-y)\psi'(y)\,dy.$$ On the other hand, we
have already seen at Example \ref{ex:fbm} that ${\mathbf
A}\psi(s)=\iott H\,|s-y|^{2H-1}\text{sign}(s-y)\psi'(y)\,dy$. That
is, on $\cald$, we have the following identity of operators:
$c_H^2\mathbf D_{+}^\alpha \mathbf D_{-}^\alpha\,=\,{\mathbf A}$.
Moreover, these  operators can be extended (and  coincide) by
density arguments to $I_{-}^{\al}(\mathcal E_H)$ (see \cite{CH} for
the definition of this space). Finally, in this case,
$\ch=I_{-}^{\al}\big(L^2([0,T]\big)$ is a subset of $L^2([0,T])$.
Using these observations,  it is readily checked that the extended
divergence operator defined above coincides with the extended
divergence given in \cite{CH}, restricted to $L^2$ spaces.
\end{example}

\subsection{Change of variable formula for Skorohod integrals}
We can now turn to the main aim of this section, namely the proof of a  change of variable formula for $f(x)$ based on our extended divergence operator $\ddi$.
  Interestingly enough, this will be achieved under some non restrictive exponential growth conditions on $f$.%, which are less restrictive than in the Stratonovich case
%  of Theorem~\ref{ito-strato-N}:

\begin{definition}\label{growuthcond}
We will say that a function $f:\R^d\to \R$ satisfies the growth
condition \gc if there exist positive constants $C$ and $\lambda$
such that
\begin{equation}\label{eq-gc-1}
\lambda<\frac1{4\,d\,\max_{t\in[0,T]}R_t},
\quad\mbox{and}\quad
|f(x)|\le C\,e^{\lambda\,|x|^2}\,\ \text{for all }x\in\R^d.
\end{equation}
\end{definition}

\smallskip

Notice that $\max_{t\in[0,T]}R_t=\max_{t\in[0,T]}E[|x_t|^2]$. Thus the growth condition above
implies that
$$\be[\max_{t\in[0,T]}|f(x_t)|^r]\le
C^r\be\Big(e^{r\lambda\max_{t\in[0,T]}|x_t|^2}\Big),$$
and this last expectation is finite (see, for instance \cite{MR-bk}, Corollary 5.4.6)
if and only if
$$r\lambda<\frac1{2\max_{t\in[0,T]}E(|x_r|^2)}=\frac1{2\,d\,\max_{t\in[0,T]}R_t}.$$
So, if condition (\ref{eq-gc-1}) is satisfied,
there exists $r> 2$ such that
\begin{equation}\label{cotamoments}
\be \Big[\max_{t\in[0,T]}|f(x_t)|^r\Big]<\infty.
\end{equation}

\smallskip

With these preliminaries in hand, we first state a Skorohod type change of variable formula for a very regular   function $f$.
\begin{proposition}\label{ito1}
Let $f\in\mathcal C^{\infty}(\R^d)$ such that $f$ and all its
derivatives satisfy the growth condition \gc (with possibly
different $\lambda$'s and $C$'s). Then, for any $0\le s<t\le T$,
$$\1_{[s,t)}(\cdot)\,\nabla f(x_{\cdot})\in\dom^*\ddi$$
and
$$\ddi\lc \1_{[s,t)}(\cdot)\,\nabla
f(x_{\cdot})\rc=f(x_t)-f(x_s)-\frac12\int_s^t \Delta f(x_\rho)\,R'_\rho  d\rho.$$
\end{proposition}
\begin{proof}
Since $f$ and all its derivatives satisfy  growth condition \gc, the process
$\1_{[s,t)}\nabla f(x)$ is an element of $\lp L^2(\oom\times[0,T])\rp^d$
and we also have
$$f(x_t)-f(x_s)-\frac12\int_s^t \Delta f(x_\rho )\,R'_\rho
d\rho  \in L^2(\oom).$$ So, we only need to show that for any $n\ge 0$
and any $\vp\in\lp\cald\rp^d$ the following equality is
satisfied:
\begin{multline}\label{eq-div-7}
\be\lc \lp f(x_t)-f(x_s)-\frac12\int_s^t \Delta
f(x_\rho)\,R'_\rho d\rho\rp\hnj\rc  \\
=\int_s^t \lla\be\lc\!
\hnju\,\nabla f(x_\rho)\rc\,, \,{\mathbf A}\vp(\rho)\rra\,d\rho.
\end{multline}
The
proof of this fact is  similar to that of \cite[Lemma 4.3]{CH},
although some technical complications arise from the  fact that here we
deal with the multidimensional case.

\smallskip

Consider thus the Gaussian kernel
\begin{equation}\label{gk}
p(\sigma,y)=\lp2\pi\si\rp^{-\frac{d}2}\exp\lp-\frac12\,\frac{|y|^2}{\si}\rp,\quad\text{for
}\si>0,\,y\in\R^d.
\end{equation}
It is a well-known fact that $\partial_{\si}p=\frac12\Delta p$. Moreover, $\be\lc g(x_t)\rc=\intd p(R_t,y)\,g(y)\,dy$ for any regular function $g:\R^d\to\R$ such that $g$ and all its
derivatives satisfy \gc.
Using these identities, we can perform the
following computations:
\begin{eqnarray}\label{eq-div-8}
\frac{d}{dt}\be\lc g(x_t)\rc&=&\frac{d}{dt}\intd
p(R_t,y)\,g(y)\,dy=\intd\frac{\partial}{\partial\si}
p(R_t,y)\,R'_t\,g(y)\,dy\nonumber\\
&=&\frac12\, R_t'\intd\Delta p(R_t,y)\,g(y)\,dy=\frac12\, R_t'\intd
p(R_t,y)\,\Delta G(y)dy\nonumber\\
&=&\frac12\,R_t'\,\be\lc \Delta g(x_t)\rc.
\end{eqnarray}
This shows that the  function $\frac{d}{dt}\be[g(x_t)]$ is defined
in all $t\in (0,T)$ and is integrable on $[0,T]$ . As a consequence,
  $\be[g(x_t)]$ is  absolutely continuous. Using
this fact and identity (\ref{eq-div-8}), we can now prove
(\ref{eq-div-7}) when $n=0$. Indeed, observe that in this case
$H_0(x)\equiv 1$ and, by definition, $H_{-1}(x)\equiv 0$. Hence, the
right-hand side of (\ref{eq-div-7}) is equal to $0$ while the
left-hand side gives:
\begin{equation*}
\be\lc  f(x_t)-f(x_s)-\frac12\int_s^t \Delta
f(x_\rho)\,R'_\rho d\rho\rc
=\int_s^t \frac{d}{d\rho}\be\lc f(x_\rho)\rc\,d\rho-\frac12 \int_s^t\be\lc \Delta
f(x_\rho)\rc\,R_\rho'\,d\rho,
\end{equation*}
and this last quantity vanishes due to (\ref{eq-div-8}).

\smallskip

Let now  $n\ge 1$. Define for $j\in\{1,\ldots,d\}$ and $t\in
[0,T]$,
\begin{equation}\label{eq:def-G-j}
G_{\vp}^j(\rho)=\int_0^\rho  {\mathbf A}\vp(j)(s)\,ds=\lla \1_{[0,\rho)}\crj,\,\vp(j)\crj\rra_\ch.
\end{equation}
Clearly, $G_{\vp}^j$ is absolutely continuous and $\lp
G_{\vp}^j\rp'={\mathbf A}\vp(j)$ (a.e). Moreover, for a  regular
function $g$ satisfying \gc together with all its  derivatives and
for any multiindex $(j_1,\ldots,j_n)\in\{1,\ldots, d\}^n$ we have
\begin{multline}\label{eq:ibp1}
\frac{d}{d\rho}\lp\be\lc g(x_\rho)\rc\gf{1}(\rho)\cdots\gf{n}(\rho)\rp
=\frac12\;\be\lc \Delta g(x_\rho)\rc R_\rho'\,\gf{1}(\rho)\cdots\gf{n}(\rho)\\
+\be\lc g(x_\rho)\rc\sum_{r=1}^n {\mathbf
A}\vp(j_r)(\rho)\lc\prod_{l:\,l\ne r} \gf{l}(\rho)\rc,
\end{multline}
where we have used (\ref{eq-div-8}). Recall now our convention (\ref{eq:convention-partial-deriv}), allowing to write $\partial_{j_1\ldots j_n}^{n}f$ for $\party{n}f$, and set $M_t^{j_1\ldots j_n}\equiv\be[\partial_{j_1\ldots j_n}^{n} f(x_\rho)] \prod_{l=1}^{n} \gf{l}(\rho)$ for $\rho\in[s,t]$.
By integrating (\ref{eq:ibp1})   from  $s$  to
$t$ and taking $g=\partial_{j_1\ldots j_n}^{n} f$ we
obtain that
\begin{multline*}
M_t^{j_1\ldots j_n}-M_s^{j_1\ldots j_n}
= \frac12 \int_s^t  \lp \be\lc\Delta\partial_{j_1\ldots j_n}^{n}f(x_\rho)\rc\,R_\rho'\,\prod_{l=1}^{n} \gf{l}(\rho) \rp \,d\rho  \\
+\,\int_s^t \lp \be\lc\partial_{j_1\ldots j_n}^{n} f(x_\rho)\rc \Big(\sum_{r=1}^n {\mathbf
A}\vp(j_r)(\rho)\prod_{l:\,l\ne r} \gf{l}(\rho)\Big) \rp\, d\rho.
\end{multline*}
Summing these expressions over all the multiindices
$(j_1,\ldots,j_n)\in\{1,\ldots,d\}^n$ and owing to the fact that
\begin{multline*}
\sum_{j_1,\ldots,j_n}\int_s^t \be\lc\partial_{j_1\ldots j_n}^{n} f(x_\rho)\rc \Big(\sum_{r=1}^n {\mathbf A}\vp(j_r)\prod_{l:\,l\ne r} \gf{l}(\rho)\Big)\, d\rho\\
=\, n\,\sum_{j_1,\ldots,j_n}\int_s^t  \be\lc\partial_{j_1\ldots j_n}^{n} f(x_\rho)\rc
\prod_{l=1}^{n-1} \gf{l}(\rho) \,{\mathbf A}\vp(j_n)(\rho)\,d\rho,
\end{multline*}
we end up with an expression of the form
\begin{multline}\label{eq-div-10}
\sum_{j_1,\ldots,j_n} \lc M_t^{j_1\ldots j_n}-M_s^{j_1\ldots j_n} \rc
= \frac12 \sum_{j_1,\ldots,j_n}\int_s^t  \be\lc\Delta\partial_{j_1\ldots j_n}^{n}f(x_\rho)\rc\,R_\rho'\,
\prod_{l=1}^{n} \gf{l}(\rho)\,d\rho\\
+\, n\,\sum_{j_1,\ldots,j_n}\int_s^t \be\lc\partial_{j_1\ldots j_n}^{n} f(x_\rho)\rc
\prod_{l=1}^{n-1} \gf{l}(\rho) \,{\mathbf A}\vp(j_n)(\rho)\,d\rho.
\end{multline}
It should be observed at this point that, as in identity (\ref{eq:def-rsums-strato}), the symmetries of the partial derivatives of $f$ play a crucial role in the proof of the current proposition. This symmetry property appears precisely in the computations above.

\smallskip

We will see now how to obtain the desired identity (\ref{eq-div-7}) from (\ref{eq-div-10}).
Indeed, it is a well known fact (see \cite{Nu-bk} again) that $\hnju\,\vp\in\dom\ddi$ and
$\ddi\lc\hnju\,\vp\rc=n\,\hnj$.
Using these last two facts, the duality relationship between $D$
and $\ddi$  and the definition (\ref{eq:def-G-j}) of $G_j^{\vp}$ we have that, for $g$ satisfying \gc
as well as its derivatives,
\begin{eqnarray*}
\be\lc \hnj g(x_t)\rc&=&\frac1{n}\be\lla\hnju\vp\,,\,Dg(x_t)\rra_{\ch}\\
&=&\frac1{n}\be\lla\hnju\vp\,,\,\1_{[0,t)}\nabla g(x_t)\rra_\ch\\
&=&\frac1{n}\be\lc\hnju\sum_{j=1}^d\int_0^t\partial_{j} g(x_t)\,{\mathbf A}\vp(j)(\rho )\,d\rho\rc\\
&=&\frac1{n}\sum_{j=1}^d\be\lc\hnju\partial_{j} g(x_t)\rc\,G_\vp^j(t).
\end{eqnarray*}
Iterating this procedure $n$ times, one ends up with the identity
\begin{equation*}
\be\lc \hnj g(x_t)\rc=
\frac1{n!}\sum_{j_1,\ldots j_n}\!\!\be\lc \partial_{j_1\ldots j_n}^{n}
g(x_t)\rc\,\gf{1}(t)\cdots\gf{n}(t).
\end{equation*}
As an application of this general calculation, we can deduce the following equalities:
\begin{eqnarray}
\be\lc \hnj f(x_t)\rc &=&\frac1{n!}\sum_{j_1,\ldots j_n}\!\!\be\lc
\partial_{j_1\ldots j_n}^{n} f(x_t)\rc\,\gf{1}(t)\cdots\gf{n}(t) \label{eq-div-11}\\
\be\lc\hnj f(x_s)\rc&=&\frac1{n!}\sum_{j_1,\ldots j_n}\!\!\be\lc \partial_{j_1\ldots j_n}^{n}
f(x_s)\rc\,\gf{1}(s)\cdots\gf{n}(s) \label{eq-div-12}\\
\be\lc\hnj \Delta
f(x_\rho)\rc&=&\frac1{n!}\sum_{j_1,\ldots j_n}\!\!\be\lc \partial_{j_1\ldots j_n}^{n} \Delta
f(x_\rho)\rc\,\gf{1}(\rho)\cdots\gf{n}(\rho),  \label{eq-div-13}
\end{eqnarray}
and
\begin{equation}\label{eq-div-14}
\be\lc\hnju \partial_{j} f(x_\rho)\rc
=\frac1{(n-1)!}\sum_{j_1,\ldots j_{n-1}}\!\!\be\lc \partial_{j_1\ldots j_{n-1} j}
f(x_\rho)\rc\,\gf{1}(\rho)\cdots G_\vp^{j_{n-1}}(\rho).
\end{equation}
Substituting now (\ref{eq-div-11})--(\ref{eq-div-14}) in
(\ref{eq-div-10}), we obtain
\begin{multline}
n!\be\lc \hnj f(x_t)\rc-n!\be\lc \hnj
f(x_s)\rc=\frac12\,n!\int_s^t\be\lc \hnj\Delta
f(x_\rho)\rc\,R'_\rho d\rho\nonumber\\
\phantom{xxxx}+n\,(n-1)!\int_s^t\sum_{j_n=1}^d\be\lc\hnju
\,\partial_{j_n} f(x_\rho)\rc \,{\mathbf A}\vp(j_n)(\rho)\,d\rho,
\end{multline} and
this is actually equality (\ref{eq-div-7}). The proof is now finished.

\end{proof}

Since the  Skorohod divergence operator is closable, we can now generalize our change of variable formula:
\begin{theorem}\label{ito2}
The conclusions of Proposition \ref{ito1} still hold true whenever $f$ is an element of $\mathcal C^{2}(\R^d)$ such that $f$ and its partial
derivatives up to second order verify the growth condition \gc.
\end{theorem}

\begin{proof}
Let $\la$ be the constant appearing in the growth condition \gc. Given $k>2\lambda$,
denote by $p_k(y)=p(\frac1{k},y)$ the Gaussian kernel defined
in (\ref{gk}) and introduce
$f_k(y)=(f*p_k)(y)$ where, as usual, $*$ denotes the convolution product.

\smallskip

We first claim that there exist $k_0\in\N$, $C'>0$ and $\lambda'$ satisfying
$\lambda<\lambda'<\frac1{4\,d\,\max_{t\in[0,T]}R_t}$,
such that
\begin{equation}\label{eq-div-20}
\sup_{k\ge k_0}|f_k(y)|\le C'\,e^{\lambda'|y|^2}.
\end{equation}
Indeed, condition \gc easily yields
\begin{multline}
|f_k(y)|\le \lp \frac{k}{2\pi}\rp^{d/2}\int_{\R^d}
|f(y-z)|\,e^{-\frac{k|z|^2}{2}}\,dz \nonumber \\
\le \lp \frac{k}{2\pi}\rp^{d/2}\,C\,\int_{\R^d}
e^{\lambda\,|y-z|^2}\,e^{-\frac{k|z|^2}{2}}\,dz =\,C\prod_{i=1}^d
\lp \sqrt{\frac{k}{2\pi}}\int_\R
e^{\lambda\,(y_i-z_i)^2}\,e^{-\frac{k
\,z_i^2}{2}}\,dz_i\rp.\nonumber
\end{multline}
On the other hand,
$$\sqrt{\frac{k}{2\pi}}\int_\R e^{\lambda\,(y_i-z_i)^2}\,e^{-\frac{k
\,z_i^2}{2}}\,dz_i\,=\,\sqrt{\frac{k}{k-2\lambda}}\,\exp\lcl\Big(\frac{\lambda\,k}{k-2\lambda}\Big)y_i^2\rcl,
$$
and $\lim_{k\to\infty}\frac{\lambda\,k}{k-2\lambda}\,=\, \lambda$. Hence, given $\lambda'\in(\lambda,\,\frac1{4\,d\,\max_{t\in[0,T]}R_t})$,
there exists $k_0\in\N$ such that for any $k\ge k_0$ the
following inequalities are satisfied:
$$\lambda<\frac{\lambda\,k}{k-2\lambda}<\lambda'<\frac1{4\,d\,\max_{t\in[0,T]}R_t}.$$
Our claim (\ref{eq-div-20}) is now easily deduced.

\smallskip

Notice that (\ref{eq-div-20}) means that for $k\ge k_0$,  $f_k$ also satisfies the growth condition \gc (with $C'$ and
$\lambda'$ substituting $C$ and $\lambda$, respectively).   Moreover, we have that
$$\be\lc \sup_{\rho\in [0,T]}\sup_{k\ge k_0}|f_k(x_\rho)|^2\rc<\infty.$$
Thanks to this inequality, as well as similar ones involving the derivatives of $f$, one can
easily see that:
\begin{enumerate}
\item $f_k(x_s)\to f(x_s)$ and $f_k(x_t)\to f(x_t)$ in $L^2(\oom)$,
\item $\int_s^t\Delta f_k(x_\rho)\,R'_\rho \,d\rho \to \int_s^t\Delta f(x_\rho)\,R'_\rho\, d\rho$ in $L^2(\oom)$ and
\item $\1_{[s,t)}\nabla
f_k(x_{\cdot}) \to \1_{[s,t)}\nabla f(x_{\cdot})$ in $\lp L^2(\oom\times
[0,T])\rp^d$.
\end{enumerate}
The result is finally obtained by applying Proposition \ref{ito1}
and the closeness of the extended operator $\ddi$ alluded to at Remark \ref{rmk:closure-ext-divergence}.

\end{proof}

\section{Representation of the Skorohod integral}
\label{sec:rep-sko}
Up to now, we have given two unrelated change of variable formulas for $f(x)$: one based on pathwise considerations (Theorem \ref{ito-strato-N}) and the other one by means of Malliavin calculus   (Theorem \ref{ito2}).  We propose now to make a link between the two formulas and integrals by means of Riemann sums.

\smallskip

Namely, let $x$ be a process generating a rough path of order $N$. We have seen at equation~(\ref{eq:def-rsums-strato}) that the Stratonovich integral  $\cj_{st}(\nabla f(x)dx)$ is given by
$\lim_{|\Pi_{st}|\to 0} S^{\Pi_{st}}$, where
$$
S^{\Pi_{st}} :=\sum_{q=0}^{n-1} \sum_{k=1}^{N} \frac{1}{k!}
\partial_{i_k\ldots i_1}^{k}f(x_{t_q}) \,
\bx^{\1}_{t_{q}, t_{q+1}}(i_1)\,\bx^{\1}_{t_{q}, t_{q+1}}(i_2)\cdots
\bx^{\1}_{t_{q}, t_{q+1}}(i_k).
$$
In a Gaussian setting, it is thus natural to think that a natural candidate for the Skorohod integral $\ddi(\nabla f(x))$ is also given by $\lim_{|\Pi_{st}|\to 0} S^{\Pi_{st},\di}$, with
\begin{equation}\label{eq:riemann-wick}
S^{\Pi_{st},\di}:= \sum_{q=0}^n \sum_{k=1}^{N} \frac{1}{k!}
\partial_{i_k\ldots i_1}^{k}f(x_{t_q}) \di \bx^{\1}_{t_{q},
t_{q+1}}(i_1) \di \cdots \di \bx^{\1}_{t_{q}, t_{q+1}}(i_k),
\end{equation}
where $\di$ denotes the Wick product. We shall see that this is indeed the case, with the following strategy:

\smallskip

\noindent \textbf{(i)} One should thus first check
that $\lim_{\Pi_{st}}S^{\Pi_{st},\di}$ exists. In order to check this convergence, we shall
use extensively Wick calculus, in order to write
\begin{equation*}
\partial_{i_k\ldots i_1}^{k}f(x_{t_q}) \di
\bx^{\1}_{t_{q}, t_{q+1}}(i_1) \di \cdots \di \bx^{\1}_{t_{q},
t_{q+1}}(i_k) =\partial_{i_k\ldots i_1}^{k}f(x_{t_q}) \,
\bx^{\1}_{t_{q}, t_{q+1}}(i_1)  \cdots \bx^{\1}_{t_{q},
t_{q+1}}(i_k) +\rho_{t_{q}, t_{q+1}},
\end{equation*}
where $\rho$ is a certain correction increment which can be computed
explicitly. Plugging this relation into (\ref{eq:riemann-wick}), we
obtain
\begin{equation}\label{eq:Spi-Spidi}
S^{\Pi_{st},\di}=S^{\Pi_{st}}+\sum_{q=0}^{n-1} \rho_{t_{q},
t_{q+1}}.
\end{equation}

\smallskip

\noindent \textbf{(ii)} Manipulating the exact expression of the remainder $\rho$, we will be able to prove that $\lim_{|\Pi_{st}|\to 0} \sum_{q=0}^{n-1} \rho^1_{t_{q}, t_{q+1}}=
-\frac12\int_s^t \Delta f(x_v) \, R'_v\,dv$. Hence, going back to (\ref{eq:Spi-Spidi}) and invoking the fact that
$S^{\Pi_{st}}$ converges to $\cj_{st}(\nabla f(x)\, dx)$, we obtain
\begin{equation*}
\lim_{|\Pi_{st}|\to 0}S^{\Pi_{st},\di}= \cj_{st}(\nabla f(x)\, dx)
-\frac12\int_s^t \Delta f(x_v) \, R_v'\,dv =[\der
f(x)]_{st}-\frac12\int_s^t \Delta f(x_v) \, R_v'\,dv.
\end{equation*}
This gives both the convergence of $S^{\Pi_{st},\di}$ and an
It\^{o}-Skorohod formula   of the form:
\begin{equation*}
[\der f(x)]_{st}=\lim_{|\Pi_{st}|\to 0}S^{\Pi_{st},\di} +
\frac12\int_s^t \Delta f(x_v) \, R'_v\,dv.
\end{equation*}

\smallskip

\noindent \textbf{(iii)} Putting together this last equality and
Theorem \ref{ito2}, it can be deduced that under
Hypotheses~\ref{hyp:cov-x} and ~\ref{l2-integr}, the limit of
$S^{\Pi_{st},\di}$ coincides with the Skorohod integral $\ddi(\nabla
f(x))$. This gives our link relating the Stratonovich integral
$\cj(\nabla f(x) \, dx)$ and the Skorohod integral  $\ddi(\nabla
f(x))$.

\smallskip

This relatively straightforward strategy being set, we turn now to the technical details of its realization. To this   end,  the main issue is obviously the computation of the corrections between Wick and ordinary products in sums like $S^{\Pi_{st},\di}$. We thus start by recalling some basic facts of Wick computations.

\subsection{Notions of Wick calculus}
We present here the notions of Wick calculus needed later on, basically following \cite{HY}. We also use extensively the notations introduced
in  Section \ref{Mall}.

\smallskip

One way to introduce Wick products on a Wiener space is to impose the relation
$$
I_n(f_n)\di I_m(g_m)=I_{n+m}(f_n\hatotimes g_m)
$$
for any $f_n\in\ch^{\hatotimes n}$ and $g_m\in\ch^{\hatotimes m}$, where the multiple integrals $I_n(f_n)$ and $I_m(g_m)$ are defined by (\ref{eq:def-mult-intg}).
If $F=\sum_{n=1}^{N_1}I_n(f_n)$ and
$G=\sum_{m=1}^{N_2}I_m(g_m)$, we define $F\di G$ by
$$F\di G=\sum_{n=1}^{N_1}\sum_{m=1}^{N_2}I_{n+m}(f_n\hatotimes g_m).$$
By a limit argument, we can then extend the Wick product to more general
random variables (see \cite{HY} for further details). In this paper, we will take the limits in the
$L^2(\Omega)$ topology.

\def\vare{{\mathcal{E}}}

\smallskip

For $f\in\ch$ we define its exponential vector $\ce(f)$ by
\begin{align*}
\vare(f)
&:=e^{\diamond I_1(f)}= \exp\left(I_1(f)-\frac{\|f\|_{\ch}^2}{2}\right)
=\exp\left(I_1(f)- \frac12 \be\left(I_1(f)
\right)^2\right)\\
&=\sum_{n=0}^{\infty}\frac1{n!} I_n(f^{\otimes
n})=\sum_{n=0}^{\infty}\frac1{n!} I_1(f)^{\di n}.
\end{align*}
In a similar way we can define the complex exponential vector of $f$
by ($\imath$ denotes here the imaginary unity)
\begin{equation}\label{eq:expansion-wick-exp}
e^{\diamond \, \imath\,I_1(f)}=\exp\left(\imath\,I_1(f)+\frac{\|f\|_{\ch}^2}{2}\right)=\sum_{n=0}^{\infty}\frac{\imath ^n}{n!} I_1(f)^{\di
n}.
\end{equation}
With these notations in hand, an important property of Wick product  is the following relation:
for any two  elements $f$ and $g$   in $\mathcal H$, we have
\begin{equation}
\vare(f)\diamond \vare(g)=\vare(f+g)\,,  \label{e.4.5}
\end{equation}
an analogous property for the complex exponential vector being also satisfied.

\smallskip

We now state a result which is a generalization of \cite[Proposition 4.8]{HY}.
\begin{proposition}\label{skoro}
Let $F\in\dom\, D^k$ and $g\in\mathcal H^{\otimes k}$. Then
\begin{enumerate}[(1)]
\item $F\diamond I_k(g)$ is well defined in $L^2(\Omega)$.
\item $Fg\in\dom\,\delta^{\di k}.$
\item $F\diamond I_k(g)=\delta^{\di k}(Fg)$.
\end{enumerate}
\end{proposition}
\begin{proof}
Let $F\in \dom\,D^k$. This implies that $F$ admits the
chaos decomposition
$F=\sum_{n=0}^{\infty} I_n(f_n)$, with
\begin{equation}\label{domdk}
\sum_{n=0}^{\infty} n^k\,n!\,\|f_n\|^2_{\mathcal H^{\hatotimes
n}}<\infty.
\end{equation}
Define then $F_N=\sum_{n=0}^{N} I_n(f_n)$. Consider also $g\in\mathcal H^{\otimes k}$. In order to check (1), we shall see that the
limit in $L^2(\Omega)$ of $F_N\di I_k(g)$ exists, as $N\to\infty$. But
$F_N \di I_k(g)=\sum_{n=0}^{N}I_{n+k}(f_n\hatotimes g),$
and the limit in $L^2(\Omega)$ of this expression exists if and
only if
$$\sum_{n=0}^{\infty} (n+k)!\,\|f_n\hatotimes g\|^2_{\mathcal
H^{\hatotimes n+k}}<\infty.$$ This last condition is clearly
satisfied thanks to (\ref{domdk}).

\smallskip

Now we will prove our claims $(2)$ and
$(3)$ for $F$ with a finite chaos decomposition and
$g=g_1\otimes\cdots\otimes g_k$, with $g_i\in\mathcal H$, for
$i=1,\ldots,k$. More precisely, we will see that for any
$G\in\textbf{S}$ the following relationship holds:
\begin{equation}\label{dual}
\be\Big[(F\diamond I_k(g))\,G\Big]=\be\big[\langle Fg,D^k
G\rangle_{\mathcal H^{\otimes k}}\Big].
\end{equation}
This will be done by an induction argument. For $k=1$, this is a
consequence of \cite[Proposition 4.8]{HY} since $F\di
I_1(g)=\delta^{\di}(Fg)$. Suppose now that (\ref{dual}) is satisfied for
$k=K$. Therefore,
\begin{eqnarray*}
\be\Big[[F\diamond I_{K+1}(g)]\,G\Big]&=&\be\Big[\big(F\di
I_1(g_1)\di
I_k(g_2\otimes\cdots\otimes g_{_{K+1}})\big)\,G\Big]\\
&=&\be\Big[\langle \big( F\di I_1(g_1)\big)g_2\otimes\cdots\otimes
g_{_{K+1}}\,,\,D^KG\rangle_{\mathcal H^{\otimes K}}\Big],
\end{eqnarray*}
where in the last equality we have used that $F\di I_1(g_1)$ has a
finite chaos expansion and the induction hypothesis. The last
expression can be rewritten as
$$\be\Big[\big(F\di I_1(g_1)\big)\,D^K_{g_{_2}\otimes\cdots\otimes
g_{_{K+1}}}G\Big].$$ Since  $D^K_{g_2\otimes\cdots\otimes
g_{{K+1}}}G\in\textbf{S}$, we can apply the case $k=1$ to the above
expression and  we obtain
\begin{align*}
&\be\Big[[F\diamond I_{K+1}(g)]\,G\Big]=
\be\Big[\langle
Fg_1\,,\,D^K_{g_{_2}\otimes\cdots\otimes
g_{_{K+1}}}G\rangle_{\mathcal H}\Big]= \be\Big[F
\,D_{_{g_1}}^1\big(D^K_{g_{_2}\otimes\cdots\otimes g_{_{K+1}}}G\big)\Big]\\
=&\be\Big[F \,D^{K+1}_{g_{_1}\otimes g_{_2}\otimes\cdots\otimes
g_{_{K+1}}}G\Big]=\be\Big[\langle F\,g_1\otimes\cdots\otimes
g_{_{K+1}}\,,\,D^{K+1}G\rangle_{\mathcal H^{\otimes K+1}}\Big],
\end{align*}
which finishes our induction procedure.
Thus, (\ref{dual}) is satisfied for $F$ with a finite chaos
expansion and $g$ a tensor product of elements of $\mathcal H$.

\smallskip

To extend the result to a general $F\in{\rm Dom} (D^k)$ and
$g\in\mathcal H^{\otimes k}$, we first consider the case
$F\in{\rm Dom} (D^k)$ and $g=g_1\otimes\cdots\otimes g_k$. In this
situation, identity (\ref{dual}) is a consequence of the fact that
this relationship holds for $F_N=\sum_{n=0}^{N} I_n(f_n)$ defined
above and of the  part (1) of the proposition. Finally, for a
general $g\in\mathcal H^{\otimes k}$, using the fact  that both
sides of (\ref{dual}) are linear in $g$, we can generalize this
identity to $g$ belonging to the linear span of elements of the form
$g_1\otimes\cdots\otimes g_k$, which is a dense subspace of
$\mathcal H^{\otimes k}$. So if $g\in\mathcal H^{\otimes k}$ and
$\{g^{M}\}_{M\in\mathbb N}$ is a sequence of elements of this linear
span of tensor products  such that $g^M\to g$ in $\mathcal
H^{\otimes k}$, one can easily see (by using (\ref{domdk})) that
$$F\di I_k(g^M)\rightarrow F\di I_k(g),$$
as $M\to\infty$ in $L^2(\Omega)$. Since
\begin{equation*}
E\Big[(F\diamond I_k(g^M))\,G\Big]=E\big[\langle Fg^M,D^k
G\rangle_{\mathcal H^{\otimes k}}\Big],
\end{equation*}
the proof  is completed  by a limiting argument.
\end{proof}

\subsection{One-dimensional case}
In order to simplify a little our presentation, we first show the identification of $\ddi(\nabla f(x))$ with $\lim_{|\Pi_{st}|\to 0} S^{\Pi_{st},\di}$ when $d=1$, that is when $x$ is a one-dimensional Gaussian process satisfying Hypothesis \ref{hyp:cov-x}. The first step in this direction is a general formula for Wick products of the form $G(X)\diamond Y^{\diamond p}$, where
$X$ and $Y$ are elements of the first chaos (see Section \ref{Mall} for a definition). Notice that the proof of this proposition is deferred to the Appendix for sake of clarity.
\begin{proposition}\label{basic} Let $g, h\in\mathcal H$ and let  $G:\mathbb R\rightarrow \mathbb R$ be differentiable up to order $p$ such that   all its  derivatives
$G^{(j)}$ are elements of $L^r(\mu_g)$ for any $j=0,\ldots,  p$ and
for some $r>2$, with $\mu_g=\cn(0,\|g\|_{\ch}^{2})$.
Define $X=I_1(g)$ and $Y=I_1(h)$. Then the Wick product $G(X) \diamond Y^{\diamond p}$ can be expressed in terms of ordinary products as
\begin{multline}\label{eq:correc-wick-1}
G(X) \diamond Y^{\diamond p} = G(X)   Y^{p}\\
+\sum_{0<l+2m\le p}
\frac{(-1)^{m+l}   p!}{2^m m! (p-2m-l)!} G^{(l)} (X) \left[\EE(XY)\right]^l
\left[\EE(Y^2)\right]^m Y^{p-2m-l}\,.
\end{multline}
\end{proposition}

\begin{example}\label{ex:wick-corrections}
In order to illustrate the kind of correction terms we obtain, let us write formula (\ref{eq:correc-wick-1}) for $p=1, 2, 3$:
\begin{eqnarray*}
G(X)\diamond Y
&=&G(X)Y -G'(X) \EE(XY)\\
G(X)\diamond Y^{\diamond 2}
&=& G(X)Y^2-G(X)\EE(Y^2)-2G'(X)\EE(XY) Y+ G''(X)\left[\EE(XY)\right]^2 \\
G(X)\diamond Y^{\diamond 3}
&=& G(X)Y^3 -3G(X) \EE(Y^2) Y+3G'(X)   \EE(XY)\EE(Y^2)\\
&&\quad - 3G'(X) \EE(XY)Y^2 +3G''(X)\left[\EE(XY)\right]^2 Y
-G'''(X)\left[\EE(XY)\right]^3 \,.
\end{eqnarray*}
\end{example}

We are now ready to state our representation of the Skorohod integral by Riemann-Wick sums:
\begin{theorem}\label{itosko1}
Let $x$ be a $1$-dimensional centered Gaussian process with
continuous covariance function fulfilling Hypotheses \ref{hyp:cov-x}
and \ref{l2-integr}, and assume that $x$ also satisfies
Hypotheses~\ref{hyp:rough-intro}. Let $f$ be a function in
$C^{2N}(\R)$  such that $f^{(k)}$ verifies the growth condition \gc
for $k=1,\ldots, 2N$. Then, the Skorohod integral $\ddi(\nabla
f(x))$ (whose existence is ensured by  Theorem \ref{ito2}) can be
represented as ${\rm a.s.}-\lim_{\Pi_{st}\to 0} S^{\Pi_{st},\di}$,
where $S^{\Pi_{st},\di}$ is defined by
$$
S^{\Pi_{st},\di}=\sum_{i=0}^{n-1} \sum_{k=1}^{N} \frac{1}{k!}
f^{(k)}(x_{t_i}) \di  {\big(\bx^{\1}_{t_{i} t_{i+1}}\big)}^{\di k}.
$$
Moreover, we have
\begin{equation}
\ddi( f'(x))= \int_s^t  f'(x_\rho) dx_\rho-\int_s^t  f'(x_\rho)  R'_\rho  d\rho
\end{equation}
\end{theorem}

\begin{proof}
As mentioned at the beginning of the section, our main task is to compute $S^{\Pi_{st},\di}$ in terms of ordinary products. This will be achieved by applying Proposition \ref{basic} to each term in the above sum, with $G=f$, $X=x_{t_i}$ and
$Y=\bx^{\1}_{t_{i} t_{i+1}}=x_{t_{i+1}}-x_{t_i}$.

\smallskip

To this end, notice first that the integrability conditions on $f$ required at Proposition~\ref{basic} are fulfilled as soon as condition \gc (see Definition \ref{growuthcond}) is met. Fix then $i\in\{0,\ldots,n-1\}$, recall that we set  $X=x_{t_i}$ and
$Y=\bx^{\1}_{t_{i} t_{i+1}}$, and consider the quantity $\cs_k^i:=\frac{1}{k!} f^{(k)}
(x_{t_i})\diamond(\bx^{\1}_{t_i t_{i+1}})^{\diamond,k}$. A direct application of Proposition \ref{basic} yields
\begin{equation*}
\sum_{k=1}^N\cs_k^i=
\sum_{k=1}^N\sum_{l+2m\le k}  \frac{(-1)^{l+m}}{2^m
m!l!(k-2m-l)!} f^{(k+l)}(X) \left[\EE(XY)\right]^l
\left[\EE(Y^2)\right]^m Y^{k-2m-l}\,.
\end{equation*}
Making a substitution $q=k+l$ and $l+m=u$,  this expression can be simplified into
\begin{eqnarray}\label{eq:simpl-S-k-i}
\sum_{k=1}^N\cs_k^i&=&  \sum_{q=1}^{2N} f^{(q)}(X)\sum_{l+m\le q/2}
\frac{(-1)^{l+m}}{2^m m!l!(q-2l-2m)!}
\left[\EE(XY)\right]^l \left[\EE(Y^2)\right]^m Y^{q-2l-2m} \notag\\
&=&  \sum_{q=1}^{2N} f^{(q)}(X) \sum_{u=0}^{\left[q/2\right]}
\frac{(-1)^{u}}{ (q-2u)!} \sum_{l+m=u} \frac{1}{    m!l! }
\left[\EE(XY)\right]^l \left[\frac{\EE(Y^2)}{2} \right]^m Y^{q-2u} \notag\\
&=&  \sum_{q=1}^{2N}  f^{(q)}(X) \sum_{u=0}^{\left[q/2\right]}
\frac{(-1)^{u}}{ (q-2u)!\,u!}
\left[\EE(XY)+\frac{\EE(Y^2)}{2} \right]^u Y^{q-2u}.
\end{eqnarray}
Moreover, recalling again that
$X=x_{t_i}$ and $Y=x_{t_{i+1}}-x_{t_i}$, it is easily seen that
\begin{equation}
\EE(XY)+\frac{\EE(Y^2)}{2} =\frac12\left[
\EE(x_{t_{i+1}}^2)-\EE(x_{t_{i}}^2)\right]=\frac12\, \der R_{t_{i}t_{i+1}}\, ,
\end{equation}
where we recall that $\der R_{t_{i}t_{i+1}}$ stands for $R_{t_{i+1}}-R_{t_{i}}$. Therefore, summing now over $i\in\{i,\ldots,n-1\}$ we get
\begin{equation}\label{descom}
S^{\Pi_{st},\di}
=\sum_{i=0}^{n-1} \sum_{k=1}^N\cs_k^i=
\sum_{q=1}^{2N} \sum_{u=0}^{\left[q/2\right]} \frac{(-1)^{u}}{
(q-2u)!\,u!\,2^u} \sum_{i=0}^{n-1}  \ct_{i}^{q,u},
\end{equation}
where the quantity $\ct_{i}^{q,u}$ is defined by
\begin{equation}\label{eq:def-T-i}
\ct_{i}^{q,u}=
f^{(q)}(x_{t_i})
\left(\der R_{t_{i}t_{i+1}} \right)^u \big(\bx^{\1}_{t_i,
t_{i+1}}\big)^{q-2u}.
\end{equation}
We now separate the study into different cases.

\smallskip

\noindent {\it Case 1:}  \   If  $u=1$ and $q-2u=0$ (namely
$q=2$), then
$$
\sum_{i=0}^{n-1}  \ct_{i}^{q,u}= - \frac12\sum_{i=0}^{n-1}  f''(x_{t_i})  \left[R_{t_{i+1}}-R_{t_i}
   \right]=-\frac12\int_s^t\Big(\sum_{i=0}^{n-1}f''(x_{t_{i}})\,\1_{[t_i,t_{i+1})}(\rho)\Big)R'_\rho d\rho,
$$
where in the last equality we resort to the fact that $R_\rho$ is absolutely
continuous (see Hypothesis \ref{hyp:cov-x}). From this expression, by a dominated convergence
argument one easily gets $\lim_{n\to\infty}\sum_{i=0}^{n-1}  \ct_{i}^{q,u}
= -\frac12\int_s^t  f''(x_\rho) R_\rho' \, d\rho$.

\smallskip

\noindent{\it Case 2:} If $u\ge 2$ or $u=1$, $q-2u\ge 1$, then $\lim_{n\to\infty}\sum_{i=0}^{n-1}  \ct_{i}^{q,u}=0$. Indeed, recalling definition (\ref{eq:def-T-i}) of $\ct_{i}^{q,u}$, we observe that
\begin{equation*}
\sum_{i=0}^{n-1}  \ct_{i}^{q,u} \le
\max_{0\le i\le n-1} \lcl |\bx^{\1}_{t_it_{i+1}}|^{q-2u} , \, |\der R_{t_{i}t_{i+1}}|^{u-1} \rcl \,
\int_s^t \Big(\sum_{i=0}^{n-1}|f^{(q)}(x_{t_{i}})|\,\1_{[t_i,t_{i+1})}(\rho)\Big)|R'(\rho)|d\rho.
\end{equation*}
In the right hand side of the above inequality, it is now easily seen that
$$
\lim_{n\to\infty}\max_{0\le i\le n-1} \lcl |\bx^{\1}_{t_it_{i+1}}|^{q-2u} , \, |\der R_{t_{i}t_{i+1}}|^{u-1} \rcl=0,
$$
while the integral term remains bounded by $ C\int_s^t |R'_\rho
|d\rho $, which is bounded by assumption. This completes the proof
of our claim.

\smallskip

\noindent {\it Case 3:}    If  $u=0$,  then
\begin{equation*}
\sum_{i=0}^{n-1}  \ct_{i}^{q,u}
=\sum_{i=0}^{n-1} \sum_{q=1}^{N}  \frac{1}{ q!}  f^{(q)}(x_{t_i})
\big(\bx^{\1}_{t_i t_{i+1}}\big)^{q}+
\sum_{i=0}^{n-1} \sum_{q=N+1}^{2N}  \frac{1}{ q!}  f^{q}(x_{t_i})   \big(\bx^{\1}_{t_i t_{i+1}}\big)^{q}\,.
\end{equation*}
Thus Theorem \ref{ito-strato-N} asserts that the first sum above converges to $\int_s^t f'(x_u) dx_u$, while it is easy to see that the second sum converges to $0$, thanks to the regularity properties of $x$.

\smallskip

Plugging now the study of our 3 cases into equation (\ref{descom}), the proof of our theorem is easily completed.

\end{proof}

\subsection{Relationship with existing results}\label{sec:comparison}
Several results exist on the convergence of Riemann-Wick sums, among which emerges \cite{NTa}, dealing with a situation which is similar to ours in the case of a one-dimensional process.

\smallskip

In order to compare our results with those of \cite{NTa}, let us specialize our situation to the case of a dyadic partition of an interval $[s,t]$ with $s<t$ (the case of a general partition is handled in \cite{NTa}, but this restriction will be more convenient for our purposes). Namely, for $n\ge 1$, we consider the partition $\Pi_{st}^{n}=\{t_k^n;0\le k\le 2^n\}$, where $t_{k}^{n}=s+k(t-s)/2^n$. For notational sake, we often write $t_{k}$ instead of $t_{k}^{n}$. We shall also restrict our study to the case of a fBm $B$, though \cite{NTa} deals with a rather general Gaussian process.

\smallskip

Let us first quote some results about weighted sums taken from
\cite{GRV,GN}:
\begin{proposition}\label{prop:weighted-sums}
Let $B$ be a one-dimensional fractional Brownian motion, whose covariance function is defined by (\ref{eq:cov-fBm}). Let $g$ be a $\cac^4$ function satisfying Hypothesis \gc together with all its derivatives.

\smallskip

\noindent
\textit{(i)} For $n\ge 1$, set
\begin{equation*}
V_n^{(2)}(g)=\sum_{k=0}^{2^n-1} g(B_{t_k}) \, \lc  \lp \bb^{\1}_{t_{k}t_{k+1}}\rp^2 - 2^{-2nH}\rc.
\end{equation*}
Then if $1/4<H<3/4$, we have
\begin{equation}\label{eq:lim-V-n-2}
\cl-\lim_{n\to\infty}n^{2H-1/2} V_n^{(2)}(g) = \si_H \int_0^{t-s} g(B_s) \, dW_s,
\end{equation}
where $\cl-\lim$ stands for a convergence in law, $\si_H$ is a positive constant depending only on $H$, and $W$ is a Brownian motion independent of $B$.

\smallskip

\noindent
\textit{(ii)} For $n\ge 1$, set
\begin{equation*}
\tilde V_n^{(3)}(g)=\sum_{k=0}^{2^n-1} g(B_{t_k}) \,  \lp \bb^{\1}_{t_{k}t_{k+1}}\rp^3.
\end{equation*}
Then if $H<1/2$ we have
$$
L^2(\oom)-\lim_{n\to\infty}n^{4H-1} \tilde V_n^{(3)}(g) =
-\frac{3}{2} \int_0^{t-s} g'(B_s) \, ds.
$$

\end{proposition}

We can now recall the main result of \cite{NTa}, to which we would like to compare our own computations, is the following:
\begin{proposition}\label{prop:cvge-wick-riemann-1}
Let $B$ be a one-dimensional fBm with Hurst parameter $1/4<H\le1/2$ and $f$ be a $\cac^4$ function satisfying Hypothesis \gc together with all its derivatives. For $0\le s< t\le T$, consider the set of dyadic partitions $\{\Pi_{st}^n; \, n\ge 1\}$ and set
\begin{equation}\label{eq:wick-riemann-sums-order-1}
\tilde S^{n,\di}=\sum_{k=0}^{2^n-1}
f^{\prime}(B_{t_k}) \di  \bb^{\1}_{t_{k} t_{k+1}}.
\end{equation}
Then $\tilde S^{n,\di}$ converges in $L^2(\oom)$ to $\ddi(f'(B))$
(which is the Skorohod integral introduced at Theorem \ref{ito2}).
\end{proposition}

\begin{proof}
Our aim here is not to reproduce the proof contained in \cite{NTa}, but to give a version compatible with our formalism. We shall focus on the case $1/4<H\le 1/3$, the other one being easier.

Consider first $0\le u< v\le T$. According to Example \ref{ex:wick-corrections}, we have
\begin{eqnarray*}
f'(B_s) \di \bb^{\1}_{uv}&=& f'(B_u) \, \bb^{\1}_{uv} - \frac12 f''(B_u) \, \be\lc B_u\, \bb^{\1}_{uv}\rc  \\
&=& f'(B_u) \, \bb^{\1}_{uv} - \frac12 f''(B_u) \, \lc  v^{2H}- u^{2H}\rc
+\frac12 f''(B_u) \, |t-s|^{2H}.
\end{eqnarray*}
In a rather artificial way, we shall recast this identity into
\begin{equation}\label{eq:correc-wick-2}
f'(B_s) \di \bb^{\1}_{uv}=
\sum_{j=1}^3 \frac{1}{j!} f^{(j)}(B_u) \, \lp \bb^{\1}_{uv}\rp^{j}- R^1_{uv}- R^2_{uv},
\end{equation}
with
\begin{equation*}
R^1_{uv}=\frac12 f''(B_u) \lc \lp \bb^{\1}_{uv}\rp^{2}- |v-u|^{2H}\rc,
\quad\mbox{and}\quad
R^2_{uv}=\frac16 f^{(3)}(B_u) \, \lp \bb^{\1}_{uv}\rp^{3}.
\end{equation*}
Plugging (\ref{eq:correc-wick-2}) into the definition of $\tilde S^{n,\di}$, we thus obtain
\begin{equation}\label{eq:tilde-Sn-Sn}
\tilde S^{n,\di}= S^{\Pi_{st}^n}- \frac12 \sum_{k=0}^{2^n-1} f''(B_{t_k}) \, \lc  t_{k+1}^{2H}- t_{k}^{2H}\rc
+\frac12 V_n^{(2)}(f")+\frac16 \tilde V_n^{(3)}(f^{(3)}).
\end{equation}
Now, invoking Proposition \ref{prop:weighted-sums}, it is readily
checked that both $V_n^{(2)}(f'')$ and $\tilde V_n^{(3)}(f^{(3)})$
converge to 0 in $L^2(\Omega)$ as $n\to\infty$. Hence
\begin{equation*}
L^2(\Omega)-\lim_{n\to\infty} \tilde S^{n,\di} =
\cj_{st}(f'(B) \, dB)- H \int_s^t f''(B_{u}) \, u^{2H-1} \, du,
\end{equation*}
which ends the proof.

\end{proof}

\smallskip

The aim of the computations above was to prove that the results of \cite{NTa} do not contradict ours for $H>1/4$. Note however the following points:

\smallskip

\noindent \emph{(i)} Having a look at Proposition
\ref{prop:cvge-wick-riemann-1}, one might think that the first order
Riemann-Wick sums $\tilde S^{n,\di}$ are always convergent in
$L^2(\oom)$. However, when $H< 1/4$, relation (\ref{eq:lim-V-n-2})
still holds true. This means that the term $V_n^{(2)}(f'')$
appearing in equation (\ref{eq:tilde-Sn-Sn}) is now divergent, due
to the fact that $2H-1/2<0$. The same kind of arguments also yield
the divergence of $\tilde V_n^{(3)}(f^{(3)})$ in
(\ref{eq:tilde-Sn-Sn}). It is thus reasonable to think that first
order Riemann-Wick sums will be divergent for $H<1/4$, justifying
our higher order expansions.

\smallskip

\noindent
\emph{(ii)} In light of Proposition \ref{prop:cvge-wick-riemann-1}, it is however possible that expansions of lower order than ours are sufficient to guarantee the convergence of sums like $S^{\Pi_{st},\di}$ in Theorem \ref{itosko1}. We haven't followed this line of investigation for sake of conciseness, but let us stress the fact that almost sure convergences of our Wick-Riemann sums are obtained in Theorem \ref{itosko1} and Theorem \ref{itosko-d} (for any sequence of partitions whose mesh tends to 0), while only $L^2(\oom)$ convergences are considered in Proposition \ref{prop:cvge-wick-riemann-1}.

\smallskip

\noindent
\emph{(iii)} It is also worth reminding that we aim at considering a general $d$-dimensional Gaussian process, while \cite{GRV,NTa} focus on 1-dimensional situations. It is an open question for us to know if the methods of the aforementioned papers could be easily adapted to a multidimensional process.

\subsection{Multidimensional case}
We shall now give the representation theorem for Skorohod's integral in the multidimensional case. Technically speaking, this will be an elaboration of the one-dimensional case, relying on tensorization and cumbersome notations.

\smallskip

We first need an analogous of Proposition \ref{basic} in the
multidimensional case, whose proof is also postponed to the Appendix. To this aim, let us introduce some additional notation: given $g_1,\,\ldots ,\,g_d\in\mathcal H$
define $\bar g=(g_1,\,\ldots,\,g_d)$ and denote by $\mu_{\bar g}$
the law in $\mathbb R^d$ of the random vector
$(I_1(g_1),\,\ldots,\,I_1(g_d))$.

\begin{proposition}\label{basic2} Using the notations  introduced  above,   let $g_1,\ldots,g_d,\,h_1\ldots,h_d\in\mathcal H$. Consider the random variables $X_1=I_1(g_1), \ldots, X_d=I_1(g_d)$, and $Y_1=I_1(h_1),
\ldots, Y_d=I_1(h_d)$. Suppose that $Y_1, \cdots, Y_d$ are independent and also that $X_j$ and $Y_k$ are independent for
$k\ne j$. Let $p=(p_1,\ldots,p_d)$ be a multiindex and set $|p|=\sum_{j=1}^{d}p_j$. Assume that $G\in\mathcal C^{|p|}(\mathbb R^d)$ is
such that $\partial^{\al}G\in L^r(\mu_{\bar g})$ for any multiindex $\alpha=(\alpha_1\ldots,\alpha_d)$  and for some $r>2$, with
$\alpha_k\le p_k$, $\,k=1,\ldots,d$. Then
$G(X_1, \cdots, X_d) \diamond Y_1^{\diamond p_1}\diamond \cdots\diamond Y_d^{\diamond
p_d}$ is well defined in $L^2(\Omega)$ and the following
formula holds:
\begin{multline}\label{eq:correc-wick-d}
G(X_1, \cdots, X_d) \diamond Y_1^{\diamond p_1}\diamond \cdots\diamond Y_d^{\diamond p_d}
=\sum_{ l_{1}+2m_1\le p_1}\!\!\cdots\!\!\sum_{ l_{d}+2m_d\le p_d}
\partial ^{l_1 ,\, \ldots,\, l_d }  G(X_1, \cdots, X_d) \\
\times\prod_{ k=1}^d\Big[\frac{(-1)^{ (m_k+ l_{k}) }  p_k! }
{  2^{m_k}m_k! l_{k}!
(p_k-2m_k-l_{k})!}\left(\EE(X_kY_k)\right)^{l_{k}}
\left(\EE(Y_k^2)\right)^{m_k}
Y_k^{p_k-2m_k-l_{k}}\Big]
 \,,
\end{multline}
where $\partial ^{j_1, \cdots, j_d}$  denotes
$\frac{\partial^{j_1+\cdots+j_d}}{\partial x_1^{j_1}\cdots\partial
x_d^{j_d}}$.
\end{proposition}

\smallskip

As in the one-dimensional case, the proposition above is the key ingredient in order to establish the following representation formula for Skorohod's integral:
\begin{theorem}\label{itosko-d}
Let $x$ be a $d$-dimensional centered Gaussian process with
continuous covariance function fulfilling Hypotheses \ref{hyp:cov-x}
and \ref{l2-integr}, and assume that $x$ also satisfies
Hypothesis~\ref{hyp:rough-intro}. Let $f$ be a function in
$C^{2N}(\R^d)$  such that $\partial_{\al} f$ verifies the growth
condition \gc for any multiindex $\al$ such that $|\al|\le 2N$. Then
the Skorohod integral $\ddi(\nabla f(x))$ (whose existence is
ensured by  Theorem \ref{ito2}) can be represented as ${\rm
a.s.}-\lim_{\Pi_{st}\to 0} S^{\Pi_{st},\di}$, where
$S^{\Pi_{st},\di}$ is defined by
$$
S^{\Pi_{st},\di}=\sum_{i=0}^{n-1} \sum_{k=1}^N\frac1{k!}\partial^{k}_{i_k,\ldots,i_1}
 f(x_{t_i}) \diamond \bx^{\1}_{t_{i} t_{i+1}}(i_1)\diamond\cdots\diamond
 \bx^{\1}_{t_{i} t_{i+1}}(i_k).
 $$
\end{theorem}

\begin{proof}
We mimic here the proof of Theorem \ref{itosko1}: decompose first $S^{\Pi_{st},\diamond}$ into $\sum_{i=0}^{n-1} \sum_{u=0}^N$ $\sum_{j_1+\cdots+j_d=u}\cs_{j_1,\ldots,j_d}^{i}$, where
\[
\cs_{j_1,\ldots,j_d}^{i}=
\frac{1}{j_1!\cdots j_d!}
\partial_{j_1, \ldots, j_d}^{u} f(x_{t_i}(1), \cdots, x_{t_i}(d) )\diamond
\big(\bx^{\1}_{t_{i} t_{i+1}}(1)\big)^{\diamond j_1}\diamond
\cdots\diamond \big(\bx^{\1}_{t_{i} t_{i+1}}(d)\big)^{\diamond j_d}\, .
\]
For a fixed $i\in\{0,\ldots,n-1\}$ and $k\in\{0,\ldots,d\}$, set now
\begin{equation*}
X_k=x_{t_i}(k),
\quad\mbox{and}\quad
Y_k=\bx^{\1}_{t_{i} t_{i+1}}(k).
\end{equation*}
As in the one dimensional case, it is readily checked that if $\partial_\al f$ satisfies the  growth condition \gc for any $|\al|\le 2N$, then the integrability conditions of Proposition \ref{basic2} are also fulfilled for $x_{t_i}=(x_{t_i}(1),\ldots,x_{t_i}(d))=(I_1(\1_{[0,t_i)}^{[1]}),\ldots,I_1(\1_{[0,t_i)}^{[d]}))$. This allows to write
\begin{multline*}
\sum_{u=1}^N \sum_{j_1+\cdots+j_d=u} \cs_{j_1,\ldots,j_d}^{i}=
\sum_{u=1}^N \sum_{j_1+\cdots+j_d=u} \,\sum_{ l_{1}+2m_1\le
j_1}\cdots\sum_{ l_{d}+2m_d\le j_d}
\partial_{l_1+j_1 , \cdots, l_d+j_d } f(X_1, \cdots, X_d) \\
\times\left( \prod_{k=1}^d\frac{(-1)^{(m_k+l_{k}) } } {
2^{m_k}m_k! l_{k}! (j_k-2m_k-l_{k})!}
\left.\EE(X_kY_k)\right)^{l_{k}} \left(\EE(Y_k^2)\right)^{m_k}
Y_k^{j_k-2m_k-l_{k}}\right).
\end{multline*}

Making substitution $l_k +j_k =q_k$ for $k=1,2, \cdots, d$, or
$j_k=q_k-l_k $,  the condition $l_k +2m_k\le j_k$ can be written as
$l_k+m_k=u_k$ with $0\le u_k\le q_k/2$ and therefore the same kind of manipulations as in (\ref{eq:simpl-S-k-i}) yield
\begin{multline*}
\sum_{u=1}^N \sum_{j_1+\cdots+j_d=u} \cs_{j_1,\ldots,j_d}^{i}
= \sum_{1\le q_1+\cdots+q_d\le 2N} \partial_{q_1 , \cdots, q_d} f(X_1, \cdots, X_d) \\
\times \prod_{k=1}^ d \left( \sum_{u_k=0}^{[q_k/2]}
   \frac{ (-1)^{ u_k  } }
{    u_k!   (q_k-2u_k )!}
\left(\EE(X_kY_k)+\frac{\EE(Y_k^2)}{2} \right)^{u_k} Y_k^{q_k-2u_k
}\right)
 \,.
\end{multline*}
Furthermore, like in the proof of Theorem \ref{itosko1}, we have $\EE(X_kY_k)+\frac{\EE(Y_k^2)}{2}=\der R_{t_{i}t_{i+1}}/2$. Summing over $i\in\{i,\ldots,n-1\}$ we thus end up with
\begin{eqnarray*}
S^{\Pi_{st},\di}&=& \sum_{i=0}^{n-1} \sum_{1\le q_1+\cdots+q_d\le N}
\partial_{q_1 , \cdots, q_d} f(x_{t_i}(1) , \cdots, x_{t_i}(d))
\prod_{k=1}^d    \frac{ 1 }
{      (q_k  )!}\big( \bx^{\1}_{t_i, t_{i+1}}(k)\big) ^{q_k  }\\
&& + \sum_{i=0}^{n-1} \,\,\sum_{N+1\le q_1+\cdots+q_d\le 2N}
\partial_{q_1 , \cdots, q_d} f(x_{t_i}(1) , \cdots, x_{t_i}(d))
  \prod_{k=1}^d  \frac{ 1 }
{      (q_k  )!} \big(\bx^{\1}_{t_i, t_{i+1}}(k)\big) ^{q_k  }\\
& &  -\frac12 \sum_{i=0}^{n-1} \sum_{k=1}^d
\partial_{kk}  ^2   f(x_{t_i}(1) , \cdots, x_{t_i}(d))
\left(R_{t_{i+1}}-R_{t_{i}}   \right)
+\tilde \Theta^\Pi_{st}\\
&:=& \Theta^{\Pi_{st}} _1+\Theta_2^{{\Pi_{st}}}+ \sum_{k=1}^d
\Theta_{3, k}^{\Pi_{st}}+\tilde \Theta^{{\Pi_{st}}}\,.
\end{eqnarray*}
In the last sum, the variables $\Theta$ correspond to the 3 cases we have distinguished in the proof of Theorem \ref{itosko1}: $\Theta_1^{{\Pi_{st}}}$ denotes the sums in which the $u_k$
are equal to $0$ and $1\le q_1+\cdots+q_d\le N$;
$\Theta_2^{{\Pi_{st}}}$ are the terms with $u_k=0$ and $N+1\le q_1+\cdots+q_d\le
2N$; $\Theta_{3, k}^{{\Pi_{st}}}$ corresponds to the terms with  $q_k=2$,
$q_j=0$ for all $j\ne k$ and $u_k=1$ (so that $u_j=0$ if $j\ne k$).
Finally,
$\tilde \Theta^{\Pi_{st}}$ denotes the sums with either
$u_1+\cdots+u_d\ge 2$ or some $u_k=1$ (and $u_j=0$ for
$j\not=k$) but $q_k\ge 3$. Referring again to the proof of Theorem \ref{itosko1}, it is then easy to argue that $\tilde
\Theta^{\Pi_{st}}$ and $\Theta^{\Pi_{st}}_2$ converge  to $0$,
$\Theta_1^{\Pi_{st}}$ converges to $\int_s^t \langle\nabla f(x_\rho), dx_\rho\rangle_{\R^d}$,
and $\sum_{k=1}^d \Theta_{3, k}^{\Pi_{st}}$ converges to $
 -\frac12 \int_s^t \Delta f(x_\rho) R_\rho'\, d\rho $.

\end{proof}

\section{Appendix}
In this Appendix, we prove Propositions \ref{basic} and
\ref{basic2}.   For this, we will need  the  following analytical lemma.
\begin{lemma}\label{analyt}
Let $\mu$ be a  finite measure on $(\mathbb R^d ,\,\mathcal
B(\mathbb R^d))$ and let $G\in\mathcal C^p(\mathbb R^d )$ be such
that $\partial^{\al}G\in L^r(\mu)$ for some $r\ge 1$ and any
multiindex $\al$ such that $|\al|:=\sum_{j=1}^d\alpha_j\le p$. Then,
there exists a sequence $(G_n)_{n\in\mathbb N}$ such that

\smallskip

\noindent
\emph{(1)} Each $G_n$ is a trigonometric
polynomial of several variables, that is
$G_n(x_1,\ldots,x_d)={\sum_{\rm{
finite}}}a_{l_1,\ldots,l_d}^ne^{\imath\xi^n_{l_1}x_1+\cdots+\imath\xi^n_{l_d}
x_d}$, and where $a^n_{l_1,\ldots,l_d}$ and $\xi_{l_j}^n$ are real
numbers.

\smallskip

\noindent
\emph{(2)} We have $\lim_{n\to\infty}\partial^{\al}G_n=\partial^{\al}G$ in $L^r(\mu)$ for any $\al$ such that $|\al|\le p$.
\end{lemma}

\begin{proof} This lemma is folklore, but we haven't been able to find it in any standard text
book. For this reason and for the sake of completeness, we  give
here the main ideas of its proof. First, given $G\in \mathcal
C^p(\mathbb R^d )$, there exists a sequence of $\mathcal
C^{\infty}(\mathbb R^d )$ functions with  compact support that
converge, jointly with their derivatives, to $G$ in $L^r(\mu)$. So,
one only needs to approximate a function $G\in \mathcal
C^{\infty}(\mathbb R^d )$ with  support contained in a rectangle of
$\mathbb R^d$, say  $K$. Moreover, given $\ep>0$,  we can suppose
that $\mu(K^c)<\ep$. For a such function, consider its Fourier
partial sums on the rectangle $K$ that converge uniformly, jointly
with their derivatives to $G$ and its derivatives. Since these
partial sums are periodic functions with the same period, their
$\sup$-norm on all $\mathbb R^d$ is the same that the $\sup$-norm on
the compact $K$. With these ingredients, the result is easily
obtained.
%It can be proved easily analogous to the proof of Theorem 3.1 of \cite{DHP}.

\end{proof}

\begin{proof}[Proof of  Proposition \ref{basic}]

We start with
$G(x)=e^{\imath\xi x}$, for an arbitrary $\xi\in\mathbb R$, which means that we wish to evaluate the Wick product $e^{\imath\xi X}\diamond Y^{\diamond p}$.

\smallskip

Recall that $X=I_1(g)$ and $Y=I_1(h)$. For $\xi,\eta\in\R$, consider the random variable
\begin{eqnarray*}
M(\xi,\eta)
&=&\exp\lp{-\frac{\xi^2}{2}\EE(X^2) }\rp \ce\lp  \imath \xi g +\imath \eta h\rp
=\exp\lp{-\frac{\xi^2}{2}\EE(X^2) }\rp \ce\lp  \imath \xi g \rp \di \ce\lp  \imath \eta h\rp  \\
&=&\exp\lp  \imath \xi X\rp \di \ce\lp  \imath \eta h\rp
=\sum_{p=0}^{\infty} \frac{\imath ^p \eta^p}{p!} \exp\lp  \imath \xi X\rp \di Y^{\di p},
\end{eqnarray*}
where we have invoked relation (\ref{e.4.5}) for the second equality and relation (\ref{eq:expansion-wick-exp}) for the last one. It is thus obvious that $e^{\imath\xi X}\diamond Y^{\diamond p}$ can be expressed as
\begin{equation*}
\frac{p!}{i^p} \times \ \hbox{the coefficient of $\eta^p$ in the expansion of } M(\xi,\eta).
\end{equation*}
We now proceed to this expansion: we have
\begin{eqnarray*}
M(\xi,\eta)&=&\exp\left\{ \imath\xi X +\imath\eta Y +\frac{\eta^2}{2}\EE(Y^2)+\xi \eta \EE(XY)\right\}\\
&=& e^{\imath\xi X}\sum_{k=0}^\infty \frac{(\imath\eta Y)^k}{k!}
\sum_{m=0}^\infty \frac{\eta^{2m} }{2^m m!} \left[\EE(Y^2) \right]^m
\sum_{l=0}^\infty \frac{\xi^l\eta^l}{l!}\left[\EE(XY)\right]^l\,.
\end{eqnarray*}
Hence, by computing the coefficient of $\eta^p$ in the above expression, it is easily checked that
\begin{align*}
&e^{\imath\xi X}\diamond Y^{\diamond p} = e^{\imath\xi X} Y^{
p}+e^{\imath\xi X} \sum_{0<l+2m\le p} \frac{\imath^{-2m-l}\xi^l
p!}{2^m m!\, l!\,(p-2m-l)!}\left[\EE(XY)\right]^l
\left[\EE(Y^2)\right]^m Y^{p-2m-l}\\
&=  e^{\imath\xi X} Y^{ p}+\sum_{0<l+2m\le p} \left(\frac{d^l}{dx^l}
e^{\imath\xi X}\right)
\frac{(-1)^{m+l}   p!}{2^m m!\, l!\,(p-2m-l)!}\left[\EE(XY)\right]^l
\left[\EE(Y^2)\right]^m Y^{p-2m-l},
\end{align*}
which is the desired formula (\ref{eq:correc-wick-1}) for $G(x)=e^{\imath\xi x}$.

\smallskip

Let us now see how to extend this relation to a more general function $G$.
By linearity, we first
obtain the result for any trigonometric polynomial $G$. Now, let $G$
be such that $G^{(j)}\in L^r(\mu_g)$ for any $j=0,\ldots,p$ and some
$r>2$. By Lemma \ref{analyt}, there exists a sequence
$(G_n)_{n\in\mathbb N}$  of trigonometric polynomials such that
$$G_n^{(j)}\rightarrow G^{(j)}\quad\text{in } L^r(\mu_g)\quad\text{for
any }j=0,\ldots, p.$$ This implies that $G(X)\in {\rm Dom}(D^p)$ and
that
$$D^j G(X)=G^{(j)}(X)g^{\otimes j}\quad\text{for any }
\,j=0,\ldots, p.$$
Indeed, $G_n(X)\in\mathbf{S}$ and
$D^j G_n(X)=G_n^{(j)}(X)g^{\otimes j}$ for any $j=0,\ldots, p$. Moreover, since
$$E[|G_n^{(j)}(X)-G^{(j)}(X)|^r]=\|G^{(j)}_n-G^{(j)}\|^r _{L^r(\mu_g)}$$
we have that
$$D^j G_n(X)\rightarrow G^{(j)}(X)g^{\otimes j}\quad \text{in } L^r (\Omega;\mathcal H^{\otimes j}).$$
Using that the $D^{(j)}$ are closed operators we obtain that
$G(X)\in {\rm Dom}(D^p)$ and that $D^j G(X)=G^{(j)}(X)g^{\otimes
j}\quad\text{for any } \,j=0,\ldots, p.$ In particular, owing to Proposition \ref{skoro} we have that
\begin{equation}\label{eq:wick-G-X-Y-p}
G(X)\diamond Y^{\diamond p}=G(X)\diamond I_p(h^{\otimes
p})=\delta^{\di p} (G(X)h^{\otimes p}).
\end{equation}

\smallskip

Let us go back now to our approximating sequence $(G_n)_{n\in\mathbb N}$. It is readily checked that  relation (\ref{eq:wick-G-X-Y-p}) also holds for any $G_n$. Thus, putting together the relation  $G_n(X)\diamond I_p(h^{\otimes
p})=\delta^{\di p} (G_n(X)h^{\otimes p})$ with equation (\ref{eq:correc-wick-1}) for a trigonometric polynomial, we get that
\begin{eqnarray}\label{www}
&&\!\!\!\!\!\!\!\!\!\!\!\!\!\!\!\!\!\!\!\!\!\delta^{\di p} (G_n(X)g^{\otimes p})\notag\\
&=&  G_n(X)   Y^{ p}\!\!+\!\!\!\!\sum_{0<l+2m\le p}\!\!\!
\frac{(-1)^{m+l}   p!}{2^m m!\,l! (p\!-2m\!-l)!} G_n^{(l)} (X) \left[\EE(XY)\right]^l
\left[\EE(Y^2)\right]^m Y^{p-2m-l}.
\end{eqnarray}
Since $G_n^{(l)}(X)\rightarrow G^{(l)}(X)$ in $L^r(\Omega)$ with
$r>2$ and the $Y^{k-2m-l}$ belong to all the $L^q(\Omega)$, the
right-hand side of (\ref{www}) converges in $L^2(\Omega)$, as
$n\to\infty$, to
$$G(X)   Y^{ p}+\sum_{0<l+2m\le
p}
\frac{(-1)^{m+l}   p!}{2^m m!\,l!\, (p-2m-l)!} G^{(l)} (X) \left[\EE(XY)\right]^l
\left[\EE(Y^2)\right]^m Y^{p-2m-l}\,.$$
Finally we obtain the general case of equation (\ref{eq:correc-wick-1}) by taking limits in both sides of equation (\ref{www}) and by resorting to the closeness of the operator $\delta^{\di p}$.

\end{proof}

\smallskip

As in the previous section, the extension of Proposition \ref{basic} to the multidimensional case is now an elaboration of the previous computations relying on some notational technicalities.

\begin{proof}[Proof of Proposition \ref{basic2}]
As for Proposition \ref{basic}, we first
consider $G(x)=e^{\imath\langle \xi,\, x\rangle}$, where $x=(x_1,\ldots,x_d)$ and
$\xi=(\xi_1,\ldots,\xi_d)$ are arbitrary vectors in $\R^d$. The extension
of the formula to a $G$ satisfying  the general integrability conditions of our hypotheses is then obtained
following the same approximation scheme as in the one-dimensional case, and is left to the reader for sake of conciseness.

\smallskip

In order to treat the case of $G(x)=e^{\imath\langle \xi,\, x\rangle}$, set
\begin{equation*}
M(\xi,\eta)=
\exp\lp  \imath\langle \xi, \, X\rangle\rp
\di \exp\lp  \imath\langle \eta, \, Y\rangle+ \frac12 \sum_{k=1}^{d}\eta_k Y_k^2\rp.
\end{equation*}
Along the same lines as for Proposition \ref{basic}, one can then identify $e^{\imath\langle \xi, X\rangle}\diamond Y_1^{\diamond p_1}\diamond\cdots \diamond Y_d^{\diamond p_d}$ with  $\frac{p_1!\cdots p_d!}{i^{p_1+\cdots+p_d}} \times
$ the coefficient of $\eta_1^{p_1} \cdots \eta_d^{p_d}$ in the expansion of $M(\xi,\eta)$. Moreover, thanks to relation (\ref{e.4.5}) and invoking the fact that $X_j$ and $Y_k$ are independent for $k\ne j$, we get that
\begin{equation*}
M(\xi,\eta)=
\exp\left( \imath\sum_{k=1}^d \xi_k X_k +\imath\sum_{k=1}^d \eta_k Y_k
+\frac12 \sum_{k=1}^d \eta_k^2\EE(Y_k^2) +\sum_{k=1}^ d \xi_k\eta_k
\EE(X_kY_k)
\right ).
\end{equation*}
Expanding now the exponential according to formula (\ref{eq:expansion-wick-exp}), we end up with
\begin{multline*}
M(\xi,\eta)= \sum_{p_1, \cdots, p_d=0}^\infty  \Big[ \sum_{
l_{1}+2m_1\le p_1}\! \cdots\!\sum_{ l_{d}+2m_d\le p_d}
\,\,\prod_{k=1}^d \Big\{ (\imath\xi_k) ^{l_{k}} e^{\imath\xi_k X_k}
\frac{\imath^{ (p_k-2m_k-2l_{k})} }
{  2^{m_k}m_k! l_{k}! (p_k-2m_k-l_{k})!}\\
\times\left(\EE(X_kY_k)\right)^{l_{k}}
\left(\EE(Y_k^2)\right)^{m_k}
 Y_k^{p_k-2m_k-l_{k}}\Big\}\Big]
\eta_1^{p_1} \cdots \eta_d^{p_d}\,.
\end{multline*}
Taking into account the fact that $\prod_{k=1}^d (\imath\xi_k)
^{l_{k}} e^{\imath\sum_{k=1}^d \xi_k X_k}=\partial ^{
l_{1},\,\ldots\,,\, l_{d}}e^{\imath\sum_{k=1}^d \xi_k X_k}$, our
formula~(\ref{eq:correc-wick-d}) is now easily deduced, which ends
the proof.

\end{proof}

\end{document}